\documentclass[11pt]{amsart}
\usepackage{amsfonts}
\usepackage{amsmath,amscd}
\usepackage{amsthm}
\usepackage{amssymb}
\usepackage{latexsym}
\usepackage{ulem}
\usepackage{dsfont}
\usepackage{braket}
\usepackage{tikz}
\usepackage{enumitem}   
\usepackage{float}
\usepackage{caption}
\usetikzlibrary{matrix}
\usepackage[linktocpage=true]{hyperref}

\numberwithin{equation}{section}

\usepackage{soul}
\usepackage{color}

\newtheorem{thm}{Theorem}[section]
\newtheorem{cor}[thm]{Corollary}
\newtheorem{lem}[thm]{Lemma}

\newtheorem{prop}[thm]{Proposition}

\newtheorem{example}[thm]{Example}
\newtheorem{defn}[thm]{Definition}
\newtheorem{defi}[thm]{Definition}
\newtheorem{conj}[thm]{Conjecture}
\newtheorem{rem}[thm]{Remark}

\numberwithin{equation}{section}
\begin{document}
	
	\newcommand{\Uq}{U_q sl_2}
	\newcommand{\Uqhat}{U_q \widehat{sl}_2}
	\newcommand{\Loop}{\mathcal{L} U_q sl_2}

	\newcommand{\beqa}{\begin{eqnarray}}
		\newcommand{\eeqa}{\end{eqnarray}}
	\newcommand{\thmref}[1]{Theorem~\ref{#1}}
	\newcommand{\secref}[1]{Sect.~\ref{#1}}
	\newcommand{\lemref}[1]{Lemma~\ref{#1}}
	\newcommand{\propref}[1]{Proposition~\ref{#1}}
	\newcommand{\corref}[1]{Corollary~\ref{#1}}
	\newcommand{\remref}[1]{Remark~\ref{#1}}
	\newcommand{\er}[1]{(\ref{#1})}
	\newcommand{\nc}{\newcommand}
	\newcommand{\rnc}{\renewcommand}

    \newcommand{\CWM}{\mathsf{W}_-^{(N)}}
	\newcommand{\CWP}{\mathsf{W}_+^{(N)}}
	\newcommand{\CG}{\mathsf{G}^{(N)}}
	\newcommand{\CtG}{\widetilde{\mathsf{G}}^{(N)}}
	
	\newcommand{\WMu}{\displaystyle{\sum_{k=0}^{N-1} P_{-k}^{(N)}(u) \tW_{-k}^{(N)}}}
	\newcommand{\WPu}{\displaystyle{\sum_{k=0}^{N-1} P_{-k}^{(N)}(u) \tW_{k+1}^{(N)}}}
	\newcommand{\Gu}{\displaystyle{\sum_{k=0}^{N-1} P_{-k}^{(N)}(u)\tG_{k+1}^{(N)}}}
	\newcommand{\tGu}{\displaystyle{\sum_{k=0}^{N-1} P_{-k}^{(N)}(u) \tilde{\tG}_{k+1}^{(N)}}}
	
	\nc{\cal}{\mathcal}
	\nc{\diag}{\mathrm{diag}}
	\nc{\goth}{\mathfrak}
	\rnc{\bold}{\mathbf}
	\renewcommand{\frak}{\mathfrak}
	\renewcommand{\Bbb}{\mathbb}

	\newcommand{\epsp}{\varepsilon_+}
	\newcommand{\epsm}{\varepsilon_-}
	\newcommand{\bepsp}{\overline{\varepsilon}_+}
	\newcommand{\bepsm}{\overline{\varepsilon}_-}

    \nc{\adj}{\sim}

    \newcommand{\bj}{\boldsymbol{\bar{\jmath}}}

	\newcommand{\bt}{{\bf {\mathsf{T}}}}
	\newcommand{\by}{{\bf {\mathsf{Y}}}}
	\newcommand{\id}{\mathrm{id}}
	\nc{\Cal}{\mathcal}
	\nc{\Xp}[1]{X^+(#1)}
	\nc{\Xm}[1]{X^-(#1)}
	\nc{\on}{\operatorname}
	\nc{\ch}{\mbox{ch}}
	\nc{\Z}{{\bold Z}}
	\nc{\J}{{\mathcal J}}
	\nc{\C}{{\bold C}}
	\nc{\Q}{{\bold Q}}
	\nc{\oC}{{\widetilde{C}}}
	\nc{\oc}{{\tilde{c}}}
	\nc{\ocI}{ \overline{\cal I}}
	\nc{\og}{{\tilde{\gamma}}}
	\nc{\lC}{{\overline{C}}}
	\nc{\lc}{{\overline{c}}}

    \nc{\cD}{{\mathcal{D}}}
	
	\nc{\tW}{\normalfont{{\mathsf{W}}}}
	\nc{\tG}{\normalfont{{\mathsf{G}}}}
    \nc{\tK}{\normalfont{{\mathsf{K}}}}
	\nc{\tZ}{\normalfont{{\mathsf{Z}}}}
	\nc{\tI}{{\mathsf{I}}}

	\nc{\tE}{{\mathsf{E}}}
	\nc{\tF}{{\mathsf{F}}}
	\nc{\tx}{{\mathsf{x}}}
	\nc{\tho}{{\mathsf{h}}}
	\nc{\tk}{{\mathsf{k}}}
	\nc{\tep}{{\bf{\cal E}}}

	\nc{\te}{{\mathsf{e}}}
	\nc{\tf}{{\mathsf{f}}}

	\nc{\odel}{{\overline{\delta}}}
	
	\def\pr#1{\left(#1\right)_\infty}  
	
	\renewcommand{\P}{{\mathcal P}}
	\nc{\N}{{\Bbb N}}
	\nc\beq{\begin{equation}}
		\nc\enq{\end{equation}}
	\nc\lan{\langle}
	\nc\ran{\rangle}
	\nc\bsl{\backslash}
	\nc\mto{\mapsto}
	\nc\lra{\leftrightarrow}
	\nc\hra{\hookrightarrow}
	\nc\sm{\smallmatrix}
	\nc\esm{\endsmallmatrix}
	\nc\sub{\subset}
	\nc\ti{\tilde}
	\nc\nl{\newline}
	\nc\fra{\frac}
	\nc\und{\underline}
	\nc\ov{\overline}
	\nc\ot{\otimes}
	
	\nc\ochi{\overline{\chi}}
	\nc\bbq{\bar{\bq}_l}
	\nc\bcc{\thickfracwithdelims[]\thickness0}
	\nc\ad{\text{\rm ad}}
	\nc\Ad{\text{\rm Ad}}
	\nc\Hom{\text{\rm Hom}}
	\nc\End{\text{\rm End}}
	\nc\Ind{\text{\rm Ind}}
	\nc\Res{\text{\rm Res}}
	\nc\Ker{\text{\rm Ker}}
	\rnc\Im{\text{Im}}
	\nc\sgn{\text{\rm sgn}}
	\nc\tr{\text{\rm tr}}
	\nc\Tr{\text{\rm Tr}}
	\nc\supp{\text{\rm supp}}
	\nc\card{\text{\rm card}}
	\nc\bst{{}^\bigstar\!}
	\nc\he{\heartsuit}
	\nc\clu{\clubsuit}
	\nc\spa{\spadesuit}
	\nc\di{\diamond}
	\nc\cW{\mathsf{W}}
	\nc\cG{\mathsf{G}}
    \nc\cK{\mathsf{K}}
    \nc\cZ{\mathsf{Z}}
	\nc\ocW{\overline{\cal W}}
	\nc\ocZ{\overline{\cal Z}}
	\nc\al{\alpha}
	\nc\bet{\beta}
	\nc\ga{\gamma}
	\nc\de{\delta}
	\nc\ep{\epsilon}
	\nc\io{\iota}
	\nc\om{\omega}
	\nc\si{\sigma}
	\rnc\th{\theta}
	\nc\ka{\kappa}
	\nc\la{\lambda}
	\nc\ze{\zeta}
	
	\nc\vp{\varpi}
	\nc\vt{\vartheta}
	\nc\vr{\varrho}
	
	\nc\odelta{\overline{\delta}}
	\nc\Ga{\Gamma}
	\nc\De{\Delta}
	\nc\Om{\Omega}
	\nc\Si{\Sigma}
	\nc\Th{\Theta}
	\nc\La{\Lambda}
	
	\nc\boa{\bold a}
	\nc\bob{\bold b}
	\nc\boc{\bold c}
	\nc\bod{\bold d}
	\nc\boe{\bold e}
	\nc\bof{\bold f}
	\nc\bog{\bold g}
	\nc\boh{\bold h}
	\nc\boi{\bold i}
	\nc\boj{\bold j}
	\nc\bok{\bold k}
	\nc\bol{\bold l}
	\nc\bom{\bold m}
	\nc\bon{\bold n}
	\nc\boo{\bold o}
	\nc\bop{\bold p}
	\nc\boq{\bold q}
	\nc\bor{\bold r}
	\nc\bos{\bold s}
	\nc\bou{\bold u}
	\nc\bov{\bold v}
	\nc\bow{\bold w}
	\nc\boz{\bold z}
	
	\nc\ba{\bold A}
	\nc\bb{\bold B}
	\nc\bc{\bold C}
	\nc\bd{\bold D}
	\nc\be{\bold E}
	\nc\bg{\bold G}
	\nc\bh{\bar{h}}
	\nc\bi{\bold I}
	\nc\bk{\bold K}
	\nc\bl{\bold L}
	\nc\bm{\bold M}
	\nc\bn{\bold N}
	\nc\bo{\bold O}
	\nc\bp{\bold P}
	\nc\bq{\bold Q}
	\nc\br{\bold R}
	\nc\bs{\bold S}
	\nc\bu{\bold U}
	\nc\bv{\bold V}
	\nc\bw{\bold W}
	\nc\bz{\bold Z}
	\nc\bx{\bold X}

	\nc\ca{\mathcal A}
	\nc\cb{\mathcal B}
	\nc\cc{\mathcal C}
	\nc\cd{\mathcal D}
	\nc\ce{\mathcal E}
	\nc\cf{\mathcal F}
	\nc\cg{\mathcal G}
	\rnc\ch{\mathcal H}
	\nc\ci{\mathcal I}
	\nc\cj{\mathcal J}
	\nc\ck{\mathcal K}
	\nc\cl{\mathcal L}
	\nc\cm{\mathcal M}
	\nc\cn{\mathcal N}
	\nc\co{\mathcal O}
	\nc\cp{\mathcal P}
	\nc\cq{\mathcal Q}
	\nc\car{\mathcal R}
	\nc\cs{\mathcal S}
	\nc\ct{\mathcal T}
	\nc\cu{\mathcal U}
	\nc\cv{\mathcal V}
	\nc\cz{\mathcal Z}
	\nc\cx{\mathcal X}
	\nc\cy{\mathcal Y}

	\nc\e[1]{E_{#1}}
	\nc\ei[1]{E_{\delta - \alpha_{#1}}}
	\nc\esi[1]{E_{s \delta - \alpha_{#1}}}
	\nc\eri[1]{E_{r \delta - \alpha_{#1}}}
	\nc\ed[2][]{E_{#1 \delta,#2}}
	\nc\ekd[1]{E_{k \delta,#1}}
	\nc\emd[1]{E_{m \delta,#1}}
	\nc\erd[1]{E_{r \delta,#1}}
	
	\nc\ef[1]{F_{#1}}
	\nc\efi[1]{F_{\delta - \alpha_{#1}}}
	\nc\efsi[1]{F_{s \delta - \alpha_{#1}}}
	\nc\efri[1]{F_{r \delta - \alpha_{#1}}}
	\nc\efd[2][]{F_{#1 \delta,#2}}
	\nc\efkd[1]{F_{k \delta,#1}}
	\nc\efmd[1]{F_{m \delta,#1}}
	\nc\efrd[1]{F_{r \delta,#1}}

	\nc\fa{\frak a}
	\nc\fb{\frak b}
	\nc\fc{\frak c}
	\nc\fd{\frak d}
	\nc\fe{\frak e}
	\nc\ff{\frak f}
	\nc\fg{\frak g}
	\nc\fh{\frak h}
	\nc\fj{\frak j}
	\nc\fk{\frak k}
	\nc\fl{\frak l}
	\nc\fm{\frak m}
	\nc\fn{\frak n}
	\nc\fo{\frak o}
	\nc\fp{\frak p}
	\nc\fq{\frak q}
	\nc\fr{\frak r}
	\nc\fs{\frak s}
	\nc\ft{\frak t}
	\nc\fu{\frak u}
	\nc\fv{\frak v}
    \nc\fw{\frak w}
	\nc\fz{\frak z}
	\nc\fx{\frak x}
	\nc\fy{\frak y}
	
	\nc\fA{\frak A}
	\nc\fB{\frak B}
	\nc\fC{\frak C}
	\nc\fD{\frak D}
	\nc\fE{\frak E}
	\nc\fF{\frak F}
	\nc\fG{\frak G}
	\nc\fH{\frak H}
	\nc\fJ{\frak J}
	\nc\fK{\frak K}
	\nc\fL{\frak L}
	\nc\fM{\frak M}
	\nc\fN{\frak N}
	\nc\fO{\frak O}
	\nc\fP{\frak P}
	\nc\fQ{\frak Q}
	\nc\fR{\frak R}
	\nc\fS{\frak S}
	\nc\fT{\frak T}
	\nc\fU{\frak U}
	\nc\fV{\frak V}
	\nc\fZ{\frak Z}
	\nc\fX{\frak X}
	\nc\fY{\frak Y}
	\nc\tfi{\ti{\Phi}}
	\nc\bF{\bold F}
	\rnc\bol{\bold 1}
	
	\nc\ua{\bold U_\A}
	
	\nc\qinti[1]{[#1]_i}
	\nc\q[1]{[#1]_q}
	\nc\xpm[2]{E_{#2 \delta \pm \alpha_#1}}  
	\nc\xmp[2]{E_{#2 \delta \mp \alpha_#1}}
	\nc\xp[2]{E_{#2 \delta + \alpha_{#1}}}
	\nc\xm[2]{E_{#2 \delta - \alpha_{#1}}}
	\nc\hik{\ed{k}{i}}
	\nc\hjl{\ed{l}{j}}
	\nc\qcoeff[3]{\left[ \begin{smallmatrix} {#1}& \\ {#2}& \end{smallmatrix}
		\negthickspace \right]_{#3}}
	\nc\qi{q}
	\nc\qj{q}
	
	\nc\ufdm{{_\ca\bu}_{\rm fd}^{\le 0}}

	
	\nc\isom{\cong} 
	
	\nc{\pone}{{\Bbb C}{\Bbb P}^1}
	\nc{\pa}{\partial}
	\def\H{\mathcal H}
	\def\L{\mathcal L}
	\nc{\F}{{\mathcal F}}
	\nc{\Sym}{{\goth S}}
	\nc{\A}{{\mathcal A}}
	\nc{\arr}{\rightarrow}
	\nc{\larr}{\longrightarrow}
	
	\nc{\ri}{\rangle}
	\nc{\lef}{\langle}
	\nc{\W}{{\mathcal W}}
	\nc{\uqatwoatone}{{U_{q,1}}(\su)}
	\nc{\uqtwo}{U_q(\goth{sl}_2)}
	\nc{\dij}{\delta_{ij}}
	\nc{\divei}{E_{\alpha_i}^{(n)}}
	\nc{\divfi}{F_{\alpha_i}^{(n)}}
	\nc{\Lzero}{\Lambda_0}
	\nc{\Lone}{\Lambda_1}
	\nc{\ve}{\varepsilon}
	\nc{\bepsilon}{\bar{\epsilon}}
	\nc{\bak}{\bar{k}}
	\nc{\phioneminusi}{\Phi^{(1-i,i)}}
	\nc{\phioneminusistar}{\Phi^{* (1-i,i)}}
	\nc{\phii}{\Phi^{(i,1-i)}}
	\nc{\Li}{\Lambda_i}
	\nc{\Loneminusi}{\Lambda_{1-i}}
	\nc{\vtimesz}{v_\ve \otimes z^m}
	
	\nc{\asltwo}{\widehat{\goth{sl}_2}}
	\nc\ag{\widehat{\goth{g}}}  
	\nc\teb{\tilde E_\boc}
	\nc\tebp{\tilde E_{\boc'}}
	
	\newcommand{\LR}{\bar{R}}
		\newcommand{\eeq}{\end{equation}}
	\newcommand{\ben}{\begin{eqnarray}}
		\newcommand{\een}{\end{eqnarray}}

	\setcounter{MaxMatrixCols}{30}
	\newcommand{\h}{\frac{1}{2}}
	\newcommand{\tha}{\frac{3}{2}}
	
	\newcommand{\bep}{\overline{\epsilon}_+}
	\newcommand{\bem}{\overline{\epsilon}_-}
	\newcommand{\bkp}{\overline{k}_+}
	\newcommand{\bkm}{\overline{k}_-}
	\newcommand{\kp}{k_+}
	\newcommand{\km}{k_-}

    \newcommand{\calW}{\mathcal{W}}
    \newcommand{\calG}{\mathcal{G}}

	\newcommand{\ds}{\mathds}
	
	\newcommand{\futt}{{\langle 23 \rangle} }
	\newcommand{\sfu}{{\langle 34 \rangle} }
	
	\newcommand{\CE}{\cal{E} }
	\newcommand{\CF}{\cal{F} }
	\newcommand{\CH}{\cal{H} }
	\newcommand{\CW}{\cal{W}}

	\newcommand{\aw}{\mathbf{aw}}
\newcommand{\saw}{\mathbf{saw}}

	\allowdisplaybreaks

	\makeatletter
	\def\@textbottom{\vskip \z@ \@plus 1pt}
	\let\@texttop\relax
	\makeatother

	\title[$\Loop$ and bivariate functions]{The quantum loop algebra of $sl_2$ and
     \\ $q$-Racah type bivariate functions}
	\author{Pascal Baseilhac}
	\author{Nicolas  Cramp\'e}
	\address{Institut Denis-Poisson CNRS/UMR 7013 - Universit\'e de Tours -
    Universit\'e d'Orl\'eans,
		Parc de Grammont, 37200 Tours, 
		FRANCE}
	\email{pascal.baseilhac@idpoisson.fr;nicolas.crampe@cnrs.fr}

	\begin{abstract}
    A unified algebraic framework for two different six--parameter families  of bivariate $q$-Racah type functions is given using the representation theory of the quantum loop algebra $\Loop$ of $sl_2$. The starting point of the analysis is two left and right coideal subalgebras of $\Loop$ and six commutative subalgebras, built from eight elements in $\Loop \otimes \Loop$.
The eight elements depend on two scalars $a,b \in {\mathbb C}^*$, and are diagonalized on a finite-dimensional vector space. 
The bivariate $q$-Racah type functions are interpreted as the overlap coefficients relating six `distinguished' eigenbases parametrized by $a,b$ 
 of the tensor product (evaluation) representations of $\Loop$ labeled by the evaluation parameters $u_1,u_2$.
Upon certain conditions on $a,b,u_1,u_2$, it is shown that a subset of pairs of elements act as tridiagonal pairs of type I (also called $q$-Racah type). 
For $u_1/u_2=1$, another subset of pairs of elements are conjectured to act as factorized Leonard pairs. 
Thus, in both cases corresponding overlap coefficients relating the various eigenbases associated with different pairs are obtained. Some of their properties are also discussed, as well as  their relation with known bivariate polynomials of Tratnik type and the rank $2$ Askey--Wilson algebra.
	\end{abstract}
	
	\maketitle
	
\vskip -0.5cm
	

	
	\tableofcontents

\section{Introduction}  
\subsection{Background} At the top of the discrete Askey scheme of classical $q$-hypergeometric orthogonal polynomials in one variable \cite{AW79,K10}, one finds
the $q$-Racah polynomials. A major progress on the subject is that all their remarkable properties (e.g. bispectrality, orthogonality) can be understood using the representation theory of the Askey–Wilson algebra $AW(3)$ \cite{Z91} and the corresponding theory of Leonard pairs \cite{T87,T04,Ter04}. From this perspective, the entries of the transition 
matrices relating the different eigenbases of elements of a given Leonard pair are, up to an overall factor, discrete orthogonal polynomials \cite{Ter04}.
As pointed out in \cite[Rem.\,2.1]{Ros07}, explicit examples of those eigenbases can be given in terms of $q$-Pochhammer's univariate polynomials on which 
the elements of the Leonard pair act as $q$-difference operators. 
Importantly for our purpose, this realization of a Leonard pair in terms of $q$-difference operators can be proven independently  
using the well-known embedding of $AW(3)$ into the finite quantum group $\Uq$ \cite{GZ93} combined with irreducible finite-dimensional representations of $\Uq$ \cite{WZ}.

Various examples of multivariate generalizations of $q$-Racah polynomials have appeared in the literature, and it is natural to expect a classification scheme as well as an interpretation of most of their properties within the framework of non-commutative algebras and their representation theory. For instance, in \cite[Sec.\,2]{Ter03} Paul Terwilliger observes that 
given the connection between Leonard pairs and the orthogonal (discrete) polynomials of the Askey scheme, it is expected that tridiagonal pairs (TD pairs) -- that arise within the representation theory of the $q$-Onsager algebra $O_q$ \cite{Ter03,B04} -- would lead to multivariate generalizations of $q$-Racah polynomials. Another class of polynomials are the ones of Tratnik type \cite{T91a,T91b}, among which one finds their $q$-analogs, the Gasper--Rahman multivariate polynomials \cite{GR05}. They enjoy bispectral properties \cite{I08}, and can be understood from the perspective of the $q$-Onsager algebra \cite{BM15}, or of the higher rank Askey--Wilson algebra and $U_qsl_2$ \cite{G25}. All together, this suggests the existence of a unified algebraic framework  underlying those apparently separated families of multivariate $q$-special functions.

 \subsection{Goal \& main results} 
The goal of this paper is to give a unified algebraic framework for two different families of bivariate $q$-Racah type functions, in relation with the representation theory of coideal subalgebras\footnote{Our starting point is among the simplest examples of quantum symmetric pairs \cite{Ko12}, that is isomorphic to the $q$-Onsager algebra introduced in \cite{Ter03} (see \cite{B04} in the context of quantum integrable models).} of the quantum loop algebra $\Loop$, the theory of TD pairs \cite{ITT99,Ter03} and factorized Leonard (FL) pairs  \cite{CZ23}.
 The main results are the following:
 
$\bullet$ Two different six-parameters families  of bivariate $q$-Racah type functions are  obtained using 
 two left/right coideal of $\Loop$ and six commutative subalgebras  of $\Loop \otimes \Loop$ identified in  Lemma \ref{lem:LoopqDG} and Proposition \ref{prop:sixp}. The two families, denoted respectively $S_{k_1,k_2}(\ell_1,\ell_2)$ and $T_{k_1,k_2}(\ell_1,\ell_2)$, are obtained by constructing six different `distinguished' bases parametrized  by some scalars $a,b \in {\mathbb C}^*$ of the two--parameter family $(u_1,u_2)$ of tensor product evaluation representations of $\Loop$, see Theorem \ref{thm:overlaps} and Figure \ref{fig:circ}.

 $\bullet$ TD pairs of $q$-Racah type are obtained  under certain conditions on $a,b,u_1,u_2$; See Propositions \ref{prop:TD} \& \ref{prop:diamA12}. Moreover, on the line $u_1/u_2=1$, TD pairs that admit a further refined structure associated with FL pairs of a new type called $B'_2$ are identified. In both cases, the overlap coefficients between eigenbases of the TD pairs and FL pairs are explicitly obtained; See Proposition \ref{prop:overTD} and Conjecture \ref{conj:FPBp2}. 

  \subsection{Outline}
      The paper is organized as follows. In Section \ref{sec:loop}, we start from the $q$-Onsager algebra $O_q$  that is embedded\footnote{Note that the embedding here considered essentially differs from the one in \cite{IT10,Ito14,Ito2}; See Remarks \ref{rem:phitilde} and \ref{rem:notcoid} in further sections.} via a map $\phi$ into $\Loop$, following \cite{BB09} (see also \cite{BVZ16}). This embedding is parametrized by two scalars $a,b \in {\mathbb C}^*$, see \eqref{def:A}, \eqref{def:B} with \eqref{eq:paramab}, that play a central role in the discussion. Using the so-called evaluation homomorphism $\Loop\rightarrow \Uq$ and (univariate
 polynomial) finite-dimensional representations of $\Uq$, we slightly upgrade the standard interpretation of the $q$-Racah polynomials: up to an overall factor, we interpret them as the overlap coefficients between
 polynomial eigenbases (in the formal variable $z$ and evaluation parameter $u\in {\mathbb C}^*$) of the two fundamental elements of $\phi(O_q)\subset \Loop$. 
In Section \ref{sec:biloop}, the previous procedure is uplifted to $\Loop \otimes \Loop$. Equipping $\Loop$ with the bialgebra structure $\Delta,\epsilon$ or  $\overline\Delta,\epsilon$, we construct six pairs of `distinguished' elements in $\Loop\otimes\Loop$. Their respective bivariate polynomial eigenbases (in the formal variables $z_1,z_2$)  parametrized by $a,b,u_1,u_2$ are derived, as well as the overlap coefficients between those bases. This leads to an interpretation of different bivariate functions  of $q$-Racah type as the overlaps between different bases of a two--parameter family $(u_1,u_2)$ of tensor product evaluation representations of $\Loop$. 
In Section \ref{sec:TDFLP}, we apply our construction to the theory of TD pairs \cite{Ter03} and FL pairs \cite{CZ23}. The conditions on the parameters $a,b,u_1,u_2$ such that we get  TD pairs and FL pairs are studied. Finally, the overlaps between respective eigenbases of TD pairs and FL pairs are obtained. Meanwhile, the rank $2$ Askey--Wilson algebra is also revisited from the new perspective. Comments and perspectives are given in the last section.

 \vspace{0.3cm}

\medskip

\noindent
{\bf Notations:}
All algebras are considered over the field of complex numbers $\mathbb{C}$.
Let $q\in\mathbb{C}^*$, we assume that $q$ is not a root of unity. The $q$-commutator  is
\beqa
\big[X,Y\big]_q=qXY-q^{-1}YX \label{def:qcom}
\eeqa
and $\big[X,Y\big]=\big[X,Y\big]_1=XY-YX$.
 The $q$-Pochhammer symbols are defined by, for $n,p\in \mathbb{N}$: 
\beqa
(z_1,\dots, z_p;q)_n=(z_1;q)_n\dots (z_p;q)_n\ , \quad
(z;q)_n&=& \prod_{k=0}^{n-1}(1-zq^k)\ ,\quad (z;q)_0:=1\ ,\label{def:qpoch}
\eeqa
and the the $q$-binomial by, for $n,\ell\in \mathbb{N}$, $\ell\leq n$:
\beqa
\left[\begin{array}{c}n\\ \ell \end{array}\right]_{q}=\frac{(q;q)_n}{(q;q)_\ell(q;q)_{n-\ell}}\ .\label{def:qbinom}
\eeqa
The $q$-hypergeometric function ${}_4\phi_3$ is defined by the following sum, for $n\in\mathbb{N}$:
\begin{align}
  {}_4\phi_3\left({{q^{-n},\;a_1, \;a_2,\; a_3 }\atop
{b_1,\;b_2, \;b_3}}\;\Bigg\vert \; q,z\right)=\sum_{k=0}^n \frac{(q^{-n},a_1,a_2,a_3;q)_k}{(q,b_1,b_2,b_3;q)_k}z^k\,.
\end{align}

\section{Quantum loop algebra $\Loop$ and $q$-Racah polynomials}\label{sec:loop}

In this section, we interpret the $q$-Racah polynomials as overlap coefficients between the eigenbases of two elements -- parametrized by a pair of scalars $a,b \in {\mathbb C}^*$ -- in a subalgebra of $\Loop$, acting on a one--parameter family of evaluation representations labeled by $u\in {\mathbb C}^*$. Under certains conditions on $a,b,u$, the two elements act on the finite-dimensional vector space as a Leonard pair, as recalled. 

\subsection{The quantum loop algebra $\Loop$ and the $q$-Onsager algebra}  The defining relations of $\Loop$ and the $q$-Onsager algebra $O_q$ are first recalled. Then, a certain embedding $O_q \subset \Loop$ that is central to our analysis is given. 
	\begin{defn}\label{def:loop}
    Define the extended Cartan matrix  $(a_{ij})_{i,j\in\{0,1\}}$ with $a_{ii} = 2, \; a_{ij} = -2$ for $i \neq j$.
		The quantum loop algebra $\Loop$ is the unital associative ${\mathbb C}$-algebra generated by the elements $e_i,f_i,k_i^{\pm 1} ; \; i \in \{0,1\}$ which satisfy the defining relations\begin{align} \label{affine01}
			&k_i e_j = q^{\frac{a_{ij}}{2}} e_j k_i \ ,  \qquad
			k_i f_j = q^{-\frac{a_{ij}}{2}} f_j k_i \ ,  \qquad
			[e_i,f_j]= \delta_{i,j} \frac{k_i^2 - k_i^{-2}}{q-q^{-1}}\ , \\
			&k_ik_i^{-1} = k_i^{-1} k_i = 1\ , \qquad k_0 k_1 =k_1 k_0=1 \ , 
		\end{align}
		with  the $q$-Serre relations:\begin{equation}\label{affine03}
			[e_i,[e_i,[e_i,e_j]_q]_{q^{-1}}]=0 \ , \qquad [f_i,[f_i,[f_i,f_j]_q]_{q^{-1}}]=0 \ ,\quad i\neq j\ .
		\end{equation}
\end{defn}

Consider the two elements
\beqa
c_0:= (q-q^{-1})^2f_0e_0 + qk_0^2 +q^{-1}k_0^{-2}\ ,\quad c_1:= (q-q^{-1})^2e_1f_1 + q^{-1}k_1^2 +qk_1^{-2}\ .
\eeqa
Noticing that the elements $e_0,f_0,k^{\pm 1}_0$ (resp. $e_1,f_1,k^{\pm 1}_1$) generate two distinct subalgebras of $\Loop$ that are isomorphic with $\Uq$, $c_0$ (resp. $c_1$) is central in one of the two subalgebras: \beqa
[c_0,x_0]=0 \quad \mbox{and}\quad  [c_1,x_1]=0 \quad \mbox{for}\quad x=e,f,k^{\pm 1}\ .\label{eq:xccom}
\eeqa

\begin{defn}
The  $q$-Onsager algebra $O_q$  is the unital associative ${\mathbb C}$-algebra with scalar $\rho\in{\mathbb C}^*$ generated by the elements  $W_0,W_1$ satisfying the so-called $q$-Dolan-Grady relations:
\begin{align}
	\lbrack  {\normalfont W}_0, \lbrack {\normalfont W}_0, \lbrack {\normalfont W}_0, {\normalfont W}_1\rbrack_q \rbrack_{q^{-1}} \rbrack &=\rho \lbrack {\normalfont W}_0, {\normalfont W}_1 \rbrack\ ,\label{qDG1}\\ 
	\lbrack {\normalfont W}_1, \lbrack {\normalfont W}_1, \lbrack {\normalfont W}_1, {\normalfont W}_0\rbrack_q \rbrack_{q^{-1}}\rbrack &= \rho  \lbrack {\normalfont W}_1, {\normalfont W}_0 \rbrack\ .\label{qDG2}
\end{align}   
\end{defn}

The $q$-Onsager algebra is known to be isomorphic with two distinct subalgebras of $\Loop$, see \cite[Prop.\,2.1]{BB09}, \cite{Ko12} and \cite[Prop.\,1.1 \& Sec.\,2]{IT10}, respectively. In this paper, we consider the embedding $O_q \subset \Loop$ given by \cite[Eq.\,(1.2)]{BB09}, with a suitable adjustment of scalars. 

\begin{prop}\label{prop:OqLoop} There exists an injective  algebra homomorphism $\phi: O_q \rightarrow \Loop$ such that:
	\beqa
	W_0 \mapsto \widehat{A} \ ,\quad W_1 \mapsto \widehat{B}\ ,\quad \rho \mapsto -(q^2-q^{-2})^2\ .\label{eq:phi}
	\eeqa
    where 
\beqa
\widehat{A} &:=& (q-q^{-1})(e_1k_1 -f_1k_1) + \alpha k_1^2 \ ,\label{def:A}\\
\widehat{B} &:=& (q-q^{-1})(f_0k_0 -e_0k_0) + \beta k_0^2\ , \qquad \alpha,\,\beta \in {\mathbb C}\ .\label{def:B}
\eeqa
\end{prop}
\begin{rem} Strictly speaking, the homomorphism $\phi$ is labeled by the two scalars $\alpha,\beta$. To avoid cluttering the notations, we simply write $\phi$ instead of $\phi_{\alpha,\beta}$.
\end{rem}
\begin{rem}\label{rem:phitilde} In \cite{IT10,Ito2,Ito14}, a different embedding $\tilde{\phi} \ : O_q \rightarrow \Loop$ is given, see \cite[Prop.\,1.1 \& Prop.\,1.13 for $\varepsilon=\varepsilon^*=1$]{IT10}. 
\end{rem}

\subsection{The quantum algebra $\Uq$ and the Askey--Wilson algebra \label{sec:AW3}}
There exists a surjective homomorphism from the $q$-Onsager algebra to the Askey--Wilson algebra $AW(3)$ \cite[Lem.\,10.2]{T11} (see \cite{CF20} for a review on $AW(3)$). From the perspective of the embedding $O_q \subset \Loop$, this homomorphism can be understood using the so-called {\it evaluation homomorphism} that sends $\widehat{A}=\phi(W_0),\;\widehat{B}=\phi(W_1)$ to elements in the quantum algebra $\Uq$, whose definition is now recalled.
\begin{defn}
 The quantum algebra $\Uq$ is generated by $E$, $F$, and $K$ with defining relations:
 \begin{align} \label{eq:Uq}
 	KE = q EK \ ,  \qquad
 	KF = q^{-1}FK \ ,  \qquad
 	[E,F]= \frac{K^2 - K^{-2}}{q-q^{-1}}\ , \qquad 
 	KK^{-1} = K^{-1} K = 1\ .
 \end{align}
 \end{defn}
 \noindent
 The Casimir element of $\Uq$ is given by:
 \beqa
 C&=& (q-q^{-1})^2EF + q^{-1}K^2 + qK^{-2}\ .\label{def:Cas}
 \eeqa
 
 The evaluation homomorphism $\mathsf{ev}_u \colon \Loop \rightarrow \Uq$, $u\in {\mathbb C}^*$, is such that: 
	\begin{equation}\label{map:eval}
		\begin{matrix}
			\mathsf{ev}_u(e_0)&=& uF\ , \qquad  \mathsf{ev}_u(f_0)&=&u^{-1}E\ , \qquad \mathsf{ev}_u(k_0)&=&K^{-1} \ , \\ 
			\mathsf{ev}_u(e_1)&=&uE\ , \qquad  \mathsf{ev}_u(f_1)&=&u^{-1}F\ , \qquad \mathsf{ev}_u(k_1)&=& K\ .
		\end{matrix}
	\end{equation}
The scalar $u$ is called the evaluation parameter.    
From \eqref{def:A},\eqref{def:B}, both elements are introduced.
	\beqa
	A &:=& \mathsf{ev}_u(\widehat{A}) = (q-q^{-1})(u EK -u^{-1}FK) + \alpha K^2 \ ,\label{eq:w0u}\\
	B &:=& \mathsf{ev}_u(\widehat{B}) =  (q-q^{-1})(u^{-1}EK^{-1} - u FK^{-1}) + \beta K^{-2} \ .\label{eq:w1u}
	\eeqa
\begin{rem}
The elements $A,\;B$ are essentially the so-called twisted primitive elements in \cite{Koor}, which coefficients $u^{\pm 1}$ of $EK^{\pm 1},FK^{\pm 1}$, are now interpreted as an evaluation parameter and its inverse.
\end{rem}

Note that the elements $A$ and $B$ satisfy the defining relations of the $q$-Onsager algebra since $\mathsf{ev}_u$ is an algebra homomorphism. Moreover, they satisfy relations 
\beqa
 \qquad\frac{ [A,[A,B]_q]_{q^{-1}}}{(q - q^{-1})^2}&=&
 \left(
 \left(\frac{u^2}{q} + \frac{q}{u^{2}}\right)C-\alpha\beta\right) A
 -(q+q^{-1})^2B
+(q + q^{-1})\left(\beta C-\alpha\left(\frac{u^2}{q} + \frac{q}{u^{2}}\right)\right),\label{eq:AW1} \\
  \frac{  [B,[B,A]_q]_{q^{-1}}}{(q - q^{-1})^2}&=&
 \left(
 \left(\frac{u^2}{q} + \frac{q}{u^{2}}\right)C-\alpha\beta\right) B
 -(q+q^{-1})^2B
+(q + q^{-1})\left(\alpha C-\beta\left(\frac{u^2}{q} + \frac{q}{u^{2}}\right)\right),\label{eq:AW2}
\eeqa
with $C$ given by \eqref{def:Cas}.
These relations are the defining relations of the (centrally extended) Askey--Wilson algebra $AW(3)$. 
This result has been demonstrated in \cite{GZ93}, and provides an embedding $AW(3)\subset \Uq$.

\subsection{Univariate $q$-difference operators and $q$-Racah polynomials.}\label{sec:qdiff} 
The $q$-Racah polynomials can be reconsidered in light of the representation theory of $\Loop$, as discussed in this section.

 Firstly, we construct a one--parameter family of  evaluation representations $V_{2j}(u)$ of $\Loop$ following \cite[Sec.\,4.2]{CP91}. Let $\mathbb{C}_{2j}[z]$, $j\in\frac{1}{2}{\mathbb N}$, be the vector space of dimension $2j+1$ of polynomials of degree less or equal to $2j$ in the variable $z$. Introduce the $q$-difference operators $T_\pm$ acting on $\mathbb{C}_{2j}[z]$ as:
\begin{align}
  T_\pm f(z)=f(q^{\pm 1} z).
\end{align}
Let ${\cal D}_z = {\mathbb C}[z,z^{-1}][T_\pm]$
denote the associative algebra of $q$-difference operators
with Laurent polynomials of $z$ as coefficients. The action of $\Uq$ generators on $\mathbb{C}_{2j}[z]$ is obtained using the homomorphism $\pi_j: \Uq \rightarrow  {\cal D}_z$:
\begin{align}\label{eq:pij}
  &\quad  K^{\pm 1}\mapsto q^{\mp j}T_\pm\ , 
  \qquad E \mapsto z\frac{q^{2j}T_- - q^{-2j}T_+}{q-q^{-1}}\ , \qquad F \mapsto -z^{-1}\frac{T_- - T_+}{q-q^{-1}}\ .
\end{align}
Note that $\mathbb{C}_{2j}[z]$ is an irreducible representation of $\Uq$.
The evaluation representation $V_{2j}(u)$ is 
then obtained by pulling back the irreducible finite-dimensional  representation ${\mathbb C}_{2j}[z]$ of $\Uq$ using $\mathsf{ev}_u$ defined in \eqref{map:eval}. By construction, $V_{2j}(u)$  is irreducible.

 Secondly, we consider the action of the elements $\widehat A$ and $\widehat B$ on $V_{2j}(u)$ and  related polynomial eigenbases. Their respective actions are given by certain $q$-difference operators, that depend on the evaluation parameter $u$.
Setting 
\begin{align}
    &\alpha=a+1/a\ , \qquad   \beta=b+1/b\ ,\qquad a,b\in{\mathbb C}^*\ ,\label{eq:paramab}
\end{align}
in \eqref{eq:w0u}, \eqref{eq:w1u},
we obtain:
\begin{align}
  &  \mathcal{A}=\pi_j(A)= u z q^j +u^{-1}z^{-1}q^{-j} -u^{-1} q^{-j}z^{-1}(1-zu q^{-j}a)(1-zu q^{-j}a^{- 1})T_+^2 \ , \\
   &  \mathcal{B}=\pi_j(B)= u^{-1} z q^{-j} +u z^{-1}q^{j} -uq^{j}z^{-1}(1-zu^{-1}q^{j}b)(1-zu^{-1}q^{j}b^{-1})T_-^2 \ .
\end{align}
 Evidently, these two operators satisfy the defining relations of $AW(3)$ i.e. \eqref{eq:AW1}, \eqref{eq:AW2} with the substitution $C \mapsto q^{2j+1}+ q^{-2j-1}$.

Now, recall that different  bases of $\mathbb{C}_{2j}[z]$ can be built in terms of $q$-Pochhammer symbols \eqref{def:qpoch}. For instance, let us first introduce the sequences of polynomials of total degree $2j$ defined by,  for $n=0,1,\dots, 2j$,
\begin{align}
\fu_n(z)&=\left(zuq^{-j}a;q^2\right)_{n}\left(zuq^{-j}/a;q^2\right)_{2j-n} \ ,\label{eq:un}\\
\fv_n(z)&=
    (zq^ju^{-1}b^{-1};q^{-2})_n(zq^ju^{-1}b;q^{-2})_{2j-n}\ .\label{eq:vn}
\end{align}
 The following  eigenvalue problems are straightforward extensions to $\Loop$ of the case for $\Uq$ solved in \cite[Sec.\,3.3.2]{WZ}:
    \begin{align}
     &   \mathcal{A} \fu_n=\left( q^{2n-2j}a+q^{2j-2n}a^{-1}\right) \fu_n\ , \label{eq:actA}\\
       &   \mathcal{B} \fv_{n}=\left( q^{2n-2j}b+q^{2j-2n}b^{-1} \right) \fv_{n}\ .\label{eq:actB}
    \end{align}
For generic values of $a,b$ and $u$, following \cite[Lem.\,3.1]{Ros07} it can be shown that $\{\fu_n\}_{n=0}^{2j}$ and $\{\fv_n\}_{n=0}^{2j}$  form bases of the evaluation representation $V_{2j}(u)$. 

An important result is that the overlap coefficients between the two respective polynomial eigenbases of ${\cal A}$ and ${\cal B}$, are expressed in terms of the $q$-Racah polynomials. As a straightforward application of  \cite[Eq.\,(2.16)]{Ros07}, we get:
\begin{align}
   \fv_{k}=\sum_{\ell=0}^{2j} R_k(\ell;a,b,u,j,q)  \fu_{\ell}\ ,\label{eq:overN1}
\end{align}
where
\begin{align}
    R_k(\ell;a,b,u,j,q) &= q^{2\ell(\ell-2j)}\left[\begin{array}{c}2j\\ \ell \end{array}\right]_{q^2}\frac{(q^{2-2j}abu^{-2};q^2)_\ell (q^{2-2j}a^{-1} b u^{-2};q^2)_{2j-\ell}(q^{2+2j-2k}a^{-1}b^{-1}u^{-2};q^2)_k  }
    {(q^{2\ell-4j}a^{2};q^2)_\ell (q^{-2\ell}a^{-2};q^2)_{2j-\ell}(q^{2-2j}a^{-1}bu^{-2};q^2)_k   }\nonumber\\
 &\times   {}_4\phi_3\left({{q^{-2k},\;q^{2k-4j}b^2, \;q^{-2\ell},\;q^{2\ell-4j}a^{2} }\atop
{q^{-4j},\;q^{-2j} ab u^2, \; q^{2-2j}abu^{-2} }}\;\Bigg\vert \; q^2,q^2\right)\ .\label{eq:Rac1}
\end{align}
The  $q$-hypergeometric function ${}_4\phi_3$ appearing in the previous relation is the $q$-Racah polynomial (see for example \cite{K10}). It has four parameters $a,b,u,j$.

Let us point out that the action of $\cal A$ (resp. $\cal B$) on the eigenfunctions of $\cal B$ (resp. $\cal A$) is irreducible tridiagonal for generic parameters $a,b,u$ (see also \cite[Rem.\,2.1]{Ros07}).  Explicitly, we have:
\beqa\label{eq:tri1}
{\cal A}\fv_n &=& {\cal A}_{n-1,n}\fv_{n-1}+{\cal A}_{n,n}\fv_n+{\cal A}_{n+1,n}\fv_{n+1} \ ,\\
\label{eq:tri2}
{\cal B}\fu_n &=&  {\cal B}_{n-1,n}\fu_{n-1}+{\cal B}_{n,n}\fu_n+{\cal B}_{n+1,n}\fu_{n+1}  \ ,
\eeqa
where
\begin{align}
\label{eq:Amn}    {\cal A}_{n-1,n}&=\frac{(1-q^{2n})(1-b^2q^{2n})(abu^2-q^{2j-2n+2})(a-u^2bq^{2n-2j-2})}{abu^2q^{2j}(1-b^2q^{4n-4j-2})(1-b^2q^{4n-4j})
} \,,\\
{\cal A}_{n+1,n}&=\frac{bq^{2n}(1-q^{2n-4j})(1-b^2q^{2n-4j})(abq^{2n-2j+2}-u^2)(au^2-bq^{2n-2j+2})}{au^2(1-b^2q^{4n-4j})(1-b^2q^{4n-4j+2})}\,,\\
\label{eq:A0n}  {\cal A}_{n,n}&=-{\cal A}_{n+1,n}-{\cal A}_{n-1,n}+\frac{\beta q^2}{u^2}+\alpha q^{-2j}-\frac{1}{bu^2q^{2n-2}}-\frac{bq^{2n-4j+2}}{u^2}
\,,
\end{align}
and
\begin{align}
\label{eq:Bmn} & {\cal B}_{n-1,n}={\cal A}_{n-1,n}\Big|_{a\leftrightarrow b}  a^2 q^{4n-2-4j} \,,\qquad
  {\cal B}_{n+1,n}={\cal A}_{n+1,n}\Big|_{a\leftrightarrow b}  \frac{q^{4j-2-4n}}{a^2} \,,\\
\label{eq:B0n}  & {\cal B}_{n,n}=-{\cal B}_{n-1,n}-{\cal B}_{n+1,n}+\frac{\alpha u^2}{q^2}+\beta q^{2j}-\frac{u^2q^{4j}}{aq^{2n+2}}-au^2q^{2n-2}\,.
\end{align}
These actions allow us to compute the recurrence relation and the difference equation satisfied by $R_k(\ell):=R_k(\ell;a,b,u,j,q) $. Indeed, let act with $\cal A$ on \eqref{eq:overN1}, one obtains, using \eqref{eq:actA} and \eqref{eq:tri1}:
\begin{align}\label{eq:rec1}
 {\cal A}_{k-1,k}  \fv_{k-1}+ {\cal A}_{k,k}  \fv_{k}+ {\cal A}_{k+1,k}  \fv_{k+1}=\sum_{\ell=0}^{2j} R_k(\ell;a,b,u,j,q)  (q^{2\ell-2j}a+q^{2j-2\ell}a^{-1})\fu_{\ell}\,.
\end{align}
Replacing $\fv_{k-1}$, $\fv_{k}$ and $\fv_{k+1}$ using \eqref{eq:overN1} and projecting on $\fu_\ell$, we prove the recurrence relation for $R_k(\ell)$
\begin{align}\label{eq:diff1}
  {\cal A}_{k-1,k}  R_{k-1}(\ell)+ {\cal A}_{k,k}  R_{k}(\ell)+ {\cal A}_{k+1,k}  R_{k+1}(\ell)=(q^{2\ell-2j}a+q^{2j-2\ell}a^{-1}) R_k(\ell) \,.  
\end{align}
With respect to the action of $\cal B$, similarly we get the difference equation
\begin{align}
    (q^{2k-2j}b+q^{2j-2k}b^{-1})R_k(\ell)={\cal B}_{\ell,\ell+1} R_k(\ell+1)+{\cal B}_{\ell,\ell} R_k(\ell)+{\cal B}_{\ell,\ell-1} R_k(\ell-1)\,.\label{eq:diffR2}
\end{align}
 Note that the recurrence relation \eqref{eq:diff1} and difference equation \eqref{eq:diffR2} above manifest the so-called 
{\it bispectrality} of the $q$-Racah polynomials.
Since both relations are inherited from  the operators $\cal A$ and ${\cal B}$ that satisfy the defining relations of $AW(3)$, we say that the bispectral algebra associated to the $q$-Racah polynomials is the Askey--Wilson algebra. Thus, we recover the result demonstrated in \cite{Z91}.

Now, recall the definition of a Leonard pair \cite[Def.\,1.1]{T01}. Identifying the conditions such that none of the eigenvalues \eqref{eq:actA}, \eqref{eq:actB} have multiplicities, nor the entries ${\cal A}_{n\pm 1,n},{\cal B}_{n\pm 1,n}$ are vanishing, for those conditions  $\cal A$ and $\cal B$ are diagonalizable and $\cal A$ (resp. $\cal B$) act as an irreducible tridiagonal matrix with respect to the eigenbasis of $\cal B$ (resp. $\cal A$). Also, observe that the pair $(\cal A =\pi_j(A), \cal B=\pi_j(B))$ satisfies the image via $\pi_j$ of the Askey-Wilson relations \eqref{eq:AW1}, \eqref{eq:AW2}. Applying  \cite[Thm.\,6.2]{T04}, it follows:

\begin{rem}\label{rem:LPcond} The pair $\cal A$ and $\cal B$ is a Leonard pair if and only if the following conditions hold:
\begin{enumerate}[label=(\roman*)]
\item $ab,a/b \notin \{u^2q^{2m-2j},u^{-2}q^{2m-2j+2}|m=0,\cdots,2j-1\}$\,,
\item $\pm a, \pm b \notin \{q^{-2j+1},\cdots,q^{2j-1}\}$\,. 
\end{enumerate}
\end{rem}

For completeness, let us finally introduce the following third sequence of polynomials, for $n=0,1,\dots, 2j$,
\begin{align}
    \fs_n(z)=\left(z u q^{-j}a;q^2\right)_{n}\left(zu^{-1}q^{j}b;q^{-2}\right)_{2j-n}\ .\label{eq:sn}
\end{align}
The actions of $\cal A$ and $\cal B$ on the vectors $\fs_n$ \eqref{eq:sn}  take the following upper/lower diagonal structure:
 \begin{align} 
    &    \mathcal{A} \fs_n=\left(q^{2n-2j}a+ q^{2j-2n}a^{-1} \right) \fs_n +(bu^{-2}q^{n-2j+2}-a^{-1}q^{-n})(q^{2j-n}-q^{n-2j})\fs_{n+1}\ ,\label{eq:actAsplit} \\
     &     \mathcal{B} \fs_n=\left(q^{2n-2j}b + q^{2j-2n}b^{-1} \right) \fs_n+(b^{-1}q^{2j-n}-au^{2}q^{n-2})(q^{n}-q^{-n})\fs_{n-1}\ .\label{eq:actBsplit}
    \end{align}
For generic values of $a,b$ and $u$,   $\{\fs_n\}_{n=0}^{2j}$ form a basis of $V_{2j}(u)$; See \cite[Lem.\,5.1 for $N=1$]{BVZ16}. 
Up to a renormalization, these vectors form the so-called split basis of the Leonard pair, see for example \cite{T05}.

\section{$\Loop \otimes \Loop$ and bivariate functions}\label{sec:biloop}

We have just seen that a subalgebra of $\Loop$ provides an algebraic setting for the $q$-Racah polynomials. In this Section, we shall show that certain subalgebras of $\Loop \otimes \Loop$ parametrized by $a,b$ provide an algebraic setting for  
two different families of bivariate functions  of $q$-Racah type, the latter being interpreted as overlap coefficients between certain eigenbases of a two--parameter family of evaluation representations labeled by $u_1,u_2$. 

\subsection{Bialgebra structure and $\Loop \otimes\Loop$}
Let us define $(\Loop,\Delta,\epsilon)$ the bialgebra with coproduct $\Delta \colon \Loop \to \Loop \otimes\Loop$ and counit $\epsilon \colon \Loop\to  {\mathbb C} $ defined by:
		\begin{align} \label{affinecoprodefk}
			&\Delta(e_i)= e_i \otimes k_i^{-1}  + k_i \otimes e_i , \; \; \Delta(f_i) =f_i \otimes k_i^{-1} + k_i \otimes f_i, \; \; \Delta(k_i^{\pm 1}) = k_i^{ \pm 1} \otimes k_i^{\pm 1} , \\
            \label{affine-counit}
			&\epsilon(e_i)=\epsilon(f_i)=0\ , \qquad  \epsilon(k_i^{\pm 1})=1 \ .
		\end{align}
Let $\sigma$ denote the permutation operator such that $\sigma(x\otimes y) = y \otimes x$ for any $x,y\in \Loop$.
Define the opposite coproduct of $\Loop$:
		\beqa
		\overline{\Delta}  = \sigma \circ \Delta\ .
		\eeqa
Then, $(\Loop,\overline\Delta,\epsilon)$ is also a bialgebra. 	

Various subalgebras of $\Loop \otimes \Loop$ can be considered. Recall that $\widehat{A},\;\widehat{B}$ generate the subalgebra $\phi(O_q)$ of $\Loop$, see Proposition~\ref{prop:OqLoop}. It is seen that the elements 
\beqa
\Delta(\widehat{A}) = (\widehat{A} - \alpha k_1^2) \otimes 1 + k_1^2 \otimes \widehat{A} \ ,\qquad
 \Delta(\widehat{B}) = (\widehat{B} - \beta k_0^2) \otimes 1 + k_0^2 \otimes \widehat{B} \ \label{eq:DelAB}
\eeqa
 belong to $\Loop \otimes \phi(O_q)$, which  shows that $(\phi(O_q),\Delta,\epsilon)$ is a left coideal subalgebra of $\Loop$. Alternatively,  $\overline\Delta(A),\overline\Delta(B)\in \phi(O_q) \otimes\Loop$, \textit{i.e.}, $(\phi(O_q),\overline\Delta,\epsilon)$ is a right coideal subalgebra of $\Loop$.
 \begin{rem}\label{rem:notcoid} Recall Remark \ref{rem:phitilde}. Using the definition of the coproduct given in \cite[Sec.\,1]{IT10}, it is readily checked that $\Delta(\tilde{\phi}(W_0)),\Delta(\tilde{\phi}(W_1)) \notin \tilde{\phi}(O_q)\otimes \Loop$. Thus, $(\tilde\phi(O_q),\Delta,\epsilon)$ and $(\tilde\phi(O_q),\overline\Delta,\epsilon)$ are neither left nor right coideal subalgebras of $\Loop$.
 \end{rem}
 
 In addition, we also introduce the four additional elements in $\Loop \otimes \Loop$ using the counit ($\epsilon(\widehat{A})=\epsilon(\widehat{B})=1$):
\beqa
(\id \otimes \epsilon)\Delta(\widehat{A}) &=& \widehat{A}  \otimes 1 \ ,\qquad 
 (\id \otimes \epsilon)\Delta(\widehat{B}) = \widehat{B}  \otimes 1\ ,\\
(\epsilon\otimes \id)\overline{\Delta}(\widehat{A}) &=& 1 \otimes \widehat{A}  \ ,\qquad
(\epsilon\otimes \id)\overline{\Delta}(\widehat{B}) = 1 \otimes \widehat{B} \ .
\eeqa
\begin{lem}\label{lem:LoopqDG}
	The pairs $(\Delta(\widehat{A}),\Delta(\widehat{B}))$ and $(\overline{\Delta}(\widehat{A}),\overline{\Delta}(\widehat{B}))$ satisfy the $q$-Dolan--Grady relations~\eqref{qDG1},~\eqref{qDG2} with $\rho \mapsto -(q^2-q^{-2})^2$.
\end{lem}	
\begin{proof} By Proposition \ref{prop:OqLoop}, and the fact that $\Delta$ (and $\overline\Delta$) are algebra homomorphisms.
\end{proof}
Moreover, six different commutative subalgebras of $\Loop \otimes \Loop$ are identified.
\begin{prop}\label{prop:sixp} Each of the six pairs $(1 \otimes \widehat{A},\Delta(\widehat{A}))$, $(\widehat{A} \otimes 1,\overline\Delta(\widehat{A}))$, $(1 \otimes \widehat{B},\Delta(\widehat{B}))$, $(\widehat{B} \otimes 1,\overline\Delta(\widehat{B}))$, $(\widehat{A} \otimes 1,1 \otimes \widehat{B})$,  $(\widehat{B}\otimes 1, 1 \otimes \widehat{A})$ is a commutative subalgebra of $\Loop \otimes \Loop$.
\end{prop}
\begin{proof}
It is clear that $[1\otimes \widehat{A},\Delta(\widehat{A})]=0$ due to the first equality in \eqref{eq:DelAB}. Also, $\big[\widehat{A} \otimes 1, 1 \otimes \widehat{B}\big] =0$ is immediate. Other equalities are shown using $\sigma$, the second equality in eq.\,\eqref{eq:DelAB} and the substitution  $\widehat{A} \rightarrow \widehat{B}$.
\end{proof}

\subsection{Generalizations of $A,B$}
 The images via the tensor product of the evaluation homomorphism \eqref{map:eval} of the eight elements
$\{\Delta(X),\overline\Delta(X),1\otimes X,X\otimes 1|X\in \widehat{A},\widehat{B}\}$
 are now discussed. For convenience,  we denote:
\begin{align}\label{eq:Duu}
 &   \Delta_{u_1,u_2}=(\mathsf{ev}_{u_1}\otimes\mathsf{ev}_{u_2}) \circ \Delta\,,
  \qquad  \overline{\Delta}_{u_1,u_2}=(\mathsf{ev}_{u_1}\otimes\mathsf{ev}_{u_2}) \circ \overline{\Delta}\,.
\end{align}
Let $u_1,u_2 \in {\mathbb C}^*$ be evaluation parameters. Let us define the eight elements in $\Uq \otimes \Uq$:
\beqa 
&&{A}_1 :=(\id \otimes \epsilon){\Delta}_{u_1,u_2}(\widehat{A}) =\left( (q-q^{-1})(u_1 EK -u_1^{-1}FK )+ \alpha K^2\right)\otimes 1\ , \label{eq:defA1}\\
&&{A}_2:=( \epsilon\otimes \id)\overline{\Delta}_{u_1,u_2}(\widehat{A})=1\otimes\left( (q-q^{-1})(u_2 EK -u_2^{-1}FK )+ \alpha K^2\right)\ , \label{eq:defA2}\\
&&{B}_1:=(\id \otimes \epsilon){\Delta}_{u_1,u_2}(\widehat{B})= \left((q-q^{-1})(u_1^{-1}EK^{-1} -u_1 FK^{-1} )+ \beta K^{-2}\right)\otimes 1\ ,\label{eq:defAs1}\\
&&B_2:= (\epsilon\otimes \id)\overline{\Delta}_{u_1,u_2}(\widehat{B})= 1\otimes \left( (q-q^{-1})(u_2^{-1}EK^{-1} -u_2 FK^{-1})+ \beta K^{-2}\right)\ ,\label{eq:defAs2}
\eeqa
and
\beqa
&&{A}_{12}:={\Delta}_{u_1,u_2}(\widehat{A})={A}_1+K_1^{2}({A}_2-\alpha)\ , \label{eq:defA12}\\
&& {B}_{12}:=\overline{\Delta}_{u_1,u_2}(\widehat{B})={B}_2+({B}_1-\beta)K_2^{-2}\ ,\label{eq:defB12}\\
&&{A}_{21}:=\overline{\Delta}_{u_1,u_2}(\widehat{A})={A}_2+({A}_1-\alpha)K_2^{2}\ , \\
&& {B}_{21}:={\Delta}_{u_1,u_2}(\widehat{B})={B}_1+K_1^{-2}({B}_2-\beta)\ .\label{eq:defAs21}
\eeqa

By Lemma \ref{lem:LoopqDG}, it follows:
\begin{cor}\label{cor:pairDG}
The pairs $(A_{12},B_{21})$ and $(A_{21},B_{12})$ satisfy the $q$-Dolan--Grady relations~\eqref{qDG1},~\eqref{qDG2} with $\rho \mapsto -(q^2-q^{-2})^2$.
\end{cor}

Moreover, by Proposition \ref{prop:sixp}:
\begin{cor}\label{cor:comm} The following commutation relations hold in $\Uq\otimes \Uq$:
    \begin{align}
    [A_2,A_{12}] =[B_{12},B_1]=[A_1,A_{21}]=[B_2,B_{21}]=[A_1,B_2]=[A_2,B_{1}]=0\,.
    \end{align}
\end{cor}

\subsection{Bivariate $q$-difference operators and $q$-Racah functions.} 

Bivariate $q$-Racah type functions can be reconsidered in light of the representation theory of $\Loop \otimes \Loop$, as shown in this section. Recall that the action of $\mathcal{A}$ and $\mathcal{B}$ on certain eigenbases of $V_{2j}(u)$ was given, see \eqref{eq:actA}, \eqref{eq:actB} and \eqref{eq:actAsplit}, \eqref{eq:actBsplit}. Our task is now to extend the analysis to a tensor product of evaluation representations of $\Loop$.

Firstly, a two--parameter family of evaluation representations 
$V_{2j_1}(u_1)\otimes V_{2j_2}(u_2)$ of $\Loop$ is constructed by pulling back $\mathbb{C}_{2j_1}[z_1]\otimes \mathbb{C}_{2j_2}[z_2]$, $j_1,j_2\in\frac{1}{2}{\mathbb N}$, i.e.\,the vector space of dimension $(2j_1 +1)\times(2j_2+1)$ of bivariate polynomials of degree less or equal to $2j_1$ (resp. $2j_2$) in the variable $z_1$ (resp. $z_2$).
Unless ratios of the evaluation parameters are certain powers of $q$, $V_{2j_1}(u_1)\otimes V_{2j_2}(u_2)$  is irreducible \cite{CP91}.

Secondly, the action of the eight elements  $A_1$, $A_2,\dots, B_{21}$ in \eqref{eq:defA1}-\eqref{eq:defAs21}  on $V_{2j_1}(u_1)\otimes V_{2j_2}(u_2)$ and  related polynomial eigenbases are studied. 
Let ${\cal D}_{z_1,z_2}$
denote the associative algebra of $q$-difference operators $T_\pm^{(1)},T_\pm^{(2)}$
with Laurent polynomials of $z_1,z_2$ as coefficients, such that
\begin{align}
  T^{(1)}_\pm f(z_1,z_2)=f(q^{\pm 1} z_1,z_2)\,,\qquad  T^{(2)}_\pm f(z_1,z_2)=f( z_1,q^{\pm 1}z_2)\ .\label{def:Tpm}
\end{align}
Consider the action of $\Uq$ generators on $\mathbb{C}_{2j_i}[z_i]$ via $\pi_{j_i}$ for $i=1,2$ as displayed  in \eqref{eq:pij}. The images via $\pi_{j_1}\otimes \pi_{j_2}$ of the eight elements  $A_1$, $A_2,\dots, B_{21}$ in \eqref{eq:defA1}-\eqref{eq:defAs21} are bivariate $q$-difference operators whose coefficients are Laurent polynomials in $z_1,z_2$.
Let us denote by a calligraphic letter each corresponding $q$-difference operator: 
\begin{align}
\mathcal{A}_1=(\pi_{j_1}\otimes \pi_{j_2})(A_1),\ \cdots, \ \mathcal{A}_{12}=(\pi_{j_1}\otimes \pi_{j_2})(A_{12}),\ \cdots \ ,\mathcal{B}_{21}=(\pi_{j_1}\otimes \pi_{j_2})(B_{21})\, .\label{eq:Acalij}
\end{align}
Explicitly, these operators read, for $i=1,2$
\begin{align}
\label{eq:defAi}  &  \mathcal{A}_i= u_i z_i q^{j_i} +u_i^{-1}z_i^{-1}q^{-j_i} -u_i^{-1} q^{-j_i}z_i^{-1}(1-z_iu_i q^{-j_i}a)(1-z_iu_i q^{-j_i}a^{- 1}){T^{(i)}_+}^2 \ , \\
   &  \mathcal{B}_i= u_i^{-1} z_i q^{-j_i} +u_i z_i^{-1}q^{j_i} -u_iq^{j_i}z_i^{-1}(1-z_iu_i^{-1}q^{j_i}b)(1-z_iu_i^{-1}q^{j_i}b^{-1}){T^{(i)}_-}^{2} \ ,
\end{align}
and 
\beqa
&&{\cal A}_{12}={\cal A}_1+q^{-2j_1}{T^{(1)}_+}^2({\cal A}_2-\alpha)\ , \qquad {\cal B}_{12}={\cal B}_2+({\cal B}_1-\beta)q^{2j_2}{T^{(2)}_-}^{2}\ ,\label{eq:defcalA12}\\
&&{\cal A}_{21}={\cal A}_2+({\cal A}_1-\alpha)q^{-2j_2}{T^{(2)}_+}^{2}\ ,\qquad {\cal B}_{21}={\cal B}_1+q^{2j_1}{T^{(1)}_-}^2({\cal B}_2-\beta)\ ,\label{eq:defcalAs21}
\eeqa
where $T^{(i)}_\pm$ act as \eqref{def:Tpm} and we recall the parametrization \eqref{eq:paramab} for $\alpha,\beta$. 

From Corollary \ref{cor:comm}, and the fact that $\pi_{j_1}\otimes \pi_{j_2}$ is an algebra homomorphism, we prove that  
the following commutation relations hold:
    \begin{align}
    [\mathcal{A}_2,\mathcal{A}_{12}] =[\mathcal{B}_{12},\mathcal{B}_1]=[\mathcal{A}_1,\mathcal{A}_{21}]=[\mathcal{B}_2,\mathcal{B}_{21}]=[\mathcal{A}_1,\mathcal{B}_2]=[\mathcal{A}_2,\mathcal{B}_{1}]=0\,.
    \end{align}
   These six different pairs of commuting operators can be visualized on the Figure~\ref{fig:circ}.
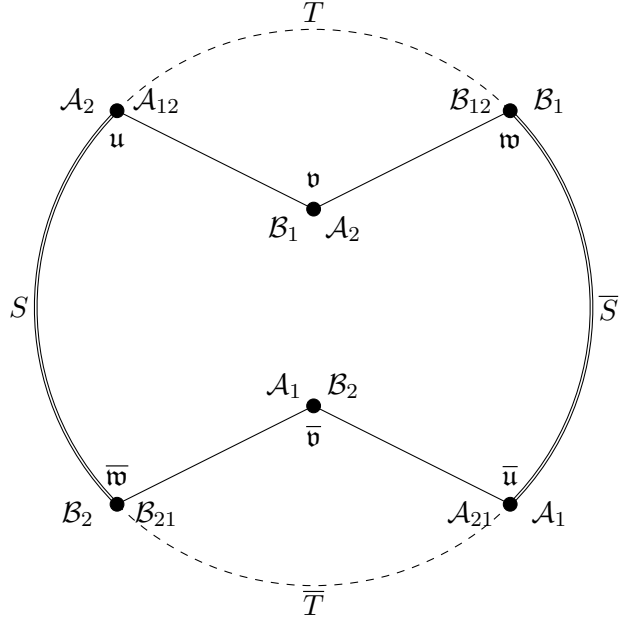
\begin{figure}[htbp]
\centering
\begin{tikzpicture}[scale=1.3]
    \def\R{4cm}\def\Rt{2.15cm}
    \coordinate (P1) at (2,2);
     \coordinate (P2) at (-2,2);
      \coordinate (P3) at (-2,-2);
       \coordinate (P4) at (2,-2);
       \coordinate (P5) at (0,1);
         \coordinate (P6) at (0,-1);
     \node at (0,3) {$T$};   
      \node at (0,-3) {$\overline{T}$};  \node at (3,0)  {$\overline{S}$};
      \node at (-3,0)  {${S}$};
   \draw[dashed] (P1) to[bend right=45] (P2);
    \draw[double] (P2) to[bend right=45] (P3);
     \draw[dashed] (P3) to[bend right=45] (P4);
      \draw[double] (P4) to[bend right=45] (P1);
     \draw (P1)--(P5);   \draw (P2)--(P5); 
     \draw (P3)--(P6);   \draw (P4)--(P6); 
      \foreach \i in {1,...,6} {\draw[fill] (P\i) circle (0.07);}
    \node at (2.4,2.1) {$\cal B_1$};  
     \node at (1.6,2.1) {$\cal B_{12}$};  
    \node at (-2.4,2.1) {$\cal A_2$};    
     \node at (-1.6,2.1) {$\cal A_{12}$};  
     \node at (-2,1.7) {$\fu$};
      \node at (2,1.7) {$\fw$};
      \node at (-2,-1.7) {$\overline \fw$};
       \node at (2,-1.7) {$\overline \fu$};
     \node at (-2.4,-2.1) {$\cal B_2$};
      \node at (2.4,-2.1) {$\cal A_1$};   
         \node at (-1.6,-2.1) {$\cal B_{21}$};
      \node at (1.6,-2.1) {$\cal A_{21}$};   
      \node at (0.3,0.8) {$\cal A_2$};
       \node at (-0.3,0.8) {$\cal B_1$};
       \node at (0,1.3) {$\fv$};
        \node at (0,-1.3) {$\overline{\fv}$};
       \node at (0.3,-0.8) {$\cal B_2$};
       \node at (-0.3,-0.8) {$\cal A_1$};
\end{tikzpicture}
\caption{Two operators around the same vertex commute. Their associated eigenbases are also indicated at the corresponding vertex. The dashed (resp. doubled) arcs are associated with overlap coefficients given by $q$-shifted products of $q$-Racah polynomials. The simple lines are associated with overlap coefficients given by $q$-Racah polynomials.\label{fig:circ}}
\end{figure}

We now turn to the construction of 
bivariate polynomial eigenbases which diagonalize simultaneously the above listed pairs of mutually commuting operators. To this aim, let us introduce the sequence of polynomials in the variables $z_1,z_2$:
\begin{align}
\label{eq:defu} &   \fu_{n_1,n_2}(z_1,z_2)=p_{n_1}(z_1u_1q^{-j_1};aq^{2n_2-2j_2},j_1,q)\;p_{n_2}(z_2u_2q^{-j_2};a,j_2,q)\ ,\\
 &   \fv_{n_1,n_2}(z_1,z_2)=p_{n_1}(z_1q^{j_1}u_1^{-1};b^{-1},j_1,q^{-1})\;p_{n_2}(z_2u_2q^{-j_2};a,j_2,q)\ ,\\
 &   \fw_{n_1,n_2}(z_1,z_2)=p_{n_1}(z_1q^{j_1}u_1^{-1};b^{-1},j_1,q^{-1})\;p_{n_2}(z_2q^{j_2}u_2^{-1};b^{-1}q^{2j_1-2n_1},j_2,q^{-1})\ ,\\
&   \overline{\fu}_{n_1,n_2}(z_1,z_2)=p_{n_1}(z_1u_1q^{-j_1};a,j_1,q)\;p_{n_2}(z_2u_2q^{-j_2};aq^{2n_1-2j_1},j_2,q)\ ,\\
 &   \overline{\fv}_{n_1,n_2}(z_1,z_2)=p_{n_1}(z_1u_1q^{-j_1};a,j_1,q)\;p_{n_2}(z_2q^{j_2}u_2^{-1};b^{-1},j_2,q^{-1})\ ,\\
 \label{eq:defbw}&   \overline{\fw}_{n_1,n_2}(z_1,z_2)=p_{n_1}(z_1q^{j_1}u_1^{-1};b^{-1}q^{2j_2-2n_2},j_1,q^{-1})\;p_{n_2}(z_2q^{j_2}u_2^{-1};b^{-1},j_2,q^{-1})\,,
\end{align}
where 
\begin{align}
    p_{n}(z;a,j,q)=\left(za;q^2\right)_{n}\left(z/a;q^2\right)_{2j-n}\ .\label{eq:defpn}
\end{align}
\begin{prop}\label{pr:eigA2B21}  The following eigenvalue problems hold for $\fu_{n_1,n_2}$,
\begin{align}
 & \mathcal{A}_{2}  \fu_{n_1,n_2}
=(q^{2n_2-2j_2}a+q^{2j_2-2n_2}a^{-1})\fu_{n_1,n_2}\ ,\label{eq:A2u}\\
   & \mathcal{A}_{12}  \fu_{n_1,n_2}
=(q^{2n_1+2n_2-2j_1-2j_2}a+q^{2j_1+2j_2-2n_1-2n_2}a^{-1})\fu_{n_1,n_2}\ ,\label{eq:specA12}
\end{align}
for $\fv_{n_1,n_2}$,
\begin{align}
    \mathcal{A}_{2}  \fv_{n_1,n_2}
=(q^{2n_2-2j_2}a+q^{2j_2-2n_2}a^{-1})\fv_{n_1,n_2}\ ,\\
 \mathcal{B}_{1}  \fv_{n_1,n_2}
=(q^{2n_1-2j_1}b+q^{2j_1-2n_1}b^{-1})\fv_{n_1,n_2}\ ,
\end{align}
for $\fw_{n_1,n_2}$,
\begin{align}
\label{eq:B1w} &\mathcal{B}_{1}  \fw_{n_1,n_2}
=(q^{2n_1-2j_1}b+q^{2j_1-2n_1}b^{-1})\fw_{n_1,n_2}\ ,\\
\label{eq:B12w} &\mathcal{B}_{12}  \fw_{n_1,n_2}
=(q^{2n_1+2n_2-2j_1-2j_2}b+q^{2j_1+2j_2-2n_1-2n_2}b^{-1})\fw_{n_1,n_2}\ ,
\end{align}
for $\overline{\fu}_{n_1,n_2}$,
\begin{align}
 & \mathcal{A}_{1}  \overline{\fu}_{n_1,n_2}
=(q^{2n_1-2j_1}a+q^{2j_1-2n_1}a^{-1})\overline{\fu}_{n_1,n_2}\ ,\\
   & \mathcal{A}_{21}  \overline{\fu}_{n_1,n_2}
=(q^{2n_1+2n_2-2j_1-2j_2}a+q^{2j_1+2j_2-2n_1-2n_2}a^{-1})\overline{\fu}_{n_1,n_2}\ ,
\end{align}
for $\overline{\fv}_{n_1,n_2}$,
\begin{align}
    \mathcal{A}_{1}  \overline{\fv}_{n_1,n_2}
=(q^{2n_1-2j_1}a+q^{2j_1-2n_1}a^{-1})\overline{\fv}_{n_1,n_2}\ ,\\
 \mathcal{B}_{2}  \overline{\fv}_{n_1,n_2}
=(q^{2n_2-2j_2}b+q^{2j_2-2n_2}b^{-1})\overline{\fv}_{n_1,n_2}\ ,
\end{align}
for $\overline{\fw}_{n_1,n_2}$,
\begin{align}
 &\mathcal{B}_{2}  \overline{\fw}_{n_1,n_2}
=(q^{2n_2-2j_2}b+q^{2j_2-2n_2}b^{-1})\overline{\fw}_{n_1,n_2}\ ,\\
 &\mathcal{B}_{21}  \overline{\fw}_{n_1,n_2}
=(q^{2n_1+2n_2-2j_1-2j_2}b+q^{2j_1+2j_2-2n_1-2n_2}b^{-1})\overline{\fw}_{n_1,n_2}\ .\label{eq:specB21}
\end{align}
\end{prop}
\proof Let us provide the proof of \eqref{eq:A2u}. Other relations are proven similarly.
The operator ${T^{(2)}_+}^{2}$ in the definition \eqref{eq:defAi} of ${\cal A}_2$ acts non--trivially only on the second factor of $\fu_{n_1,n_2}(z_1,z_2)$ in \eqref{eq:defu}:
\begin{align}
\nonumber    {T^{(2)}_+}^{2}p_{n_2}(z_2u_2q^{-j_2};a,j_2,q)&=p_{n_2}(z_2u_2q^{-j_2+2};a,j_2,q)\\
 \nonumber   &=(z_2u_2q^{-j_2+2}a;q^2)_{n_2}(z_2u_2q^{-j_2+2}/a;q^2)_{2j_2-n_2}\\
    &=\frac{(1-z_2u_2q^{2n_2-j_2}a)(1-z_2u_2q^{3j_2-2n_2}/a)}{(1-z_2u_2q^{-j_2}a)(1-z_2u_2q^{-j_2}/a)}p_{n_2}(z_2u_2q^{-j_2};a,j_2,q)\,.
\end{align}
 Using the previous result together with the explicit form \eqref{eq:defAi} of ${\cal A}_2$, we prove \eqref{eq:A2u} after simplifications.
\endproof

   Assume the parameters $a,b$  are such that the couples of eigenvalues in the previous proposition have no multiplicities. Then, the bivariate polynomials \eqref{eq:defu}-\eqref{eq:defbw} are linearly independent.  It follows that
 \beqa
 \{\fu_{n_1,n_2},\; \fv_{n_1,n_2},\;\fw_{n_1,n_2},\;\overline{\fu}_{n_1,n_2},\;\overline{\fv}_{n_1,n_2},\;\overline{\fw}_{n_1,n_2}\}_{n_1,n_2=0}^{2j_1,2j_2} \label{eq:6bases}
 \eeqa
 form six `distinguished' bases of $V_{2j_1}(u_1)\otimes V_{2j_2}(u_2)$.
These different eigenbases associated with each pair of mutually commuting operators can be seen on Figure~\ref{fig:circ}.

We obtain our first main result.

\begin{thm}\label{thm:overlaps}  The overlap coefficients between the  polynomial eigenbases \eqref{eq:defu}-\eqref{eq:defbw}  of the pairs $({\cal A}_2,{\cal A}_{12})$, $({\cal A}_2,{\cal B}_{1})$, $({\cal B}_1,{\cal B}_{12})$, $({\cal A}_1,{\cal A}_{21})$, $({\cal A}_1,{\cal B}_{2})$ and $({\cal B}_2,{\cal B}_{21})$ are given by
\begin{align}
& \fw_{k_1,k_2}  = \sum_{\ell_1=0}^{2j_1}\;\sum_{\ell_2=0}^{2j_2}
  T_{k_1,k_2}(\ell_1,\ell_2)\fu_{\ell_1,\ell_2}\ ,\label{eq:overwu} \\
&   \overline{\fw}_{k_1,k_2}  = \sum_{\ell_1=0}^{2j_1}\;\sum_{\ell_2=0}^{2j_2}
  S_{k_1,k_2}(\ell_1,\ell_2)\fu_{\ell_1,\ell_2}\ ,\label{eq:overwbaru}\\
   & \overline{\fw}_{k_1,k_2}  =\sum_{\ell_1=0}^{2j_1}\;\sum_{\ell_2=0}^{2j_2}
  \overline{T}_{k_1,k_2}(\ell_1,\ell_2)\overline{\fu}_{\ell_1,\ell_2}\ ,\\
  & \fw_{k_1,k_2}  = \sum_{\ell_1=0}^{2j_1}\;\sum_{\ell_2=0}^{2j_2}
  \overline{S}_{k_1,k_2}(\ell_1,\ell_2) \overline{\fu}_{\ell_1,\ell_2}\ ,\label{eq:overwubar}
\end{align}
where
\begin{align}
 &  T_{k_1,k_2}(\ell_1,\ell_2)= R_{k_1}(\ell_1;aq^{2\ell_2-2j_2},b,u_1,j_1,q) R_{k_2}(\ell_2;a,bq^{2k_1-2j_1},u_2,j_2,q)\ ,\label{eq:T} \\
 &  S_{k_1,k_2}(\ell_1,\ell_2)= R_{k_1}(\ell_1;aq^{2\ell_2-2j_2},bq^{2k_2-2j_2},u_1,j_1,q) R_{k_2}(\ell_2;a,b,u_2,j_2,q)\ ,\label{eq:S}\\
  &  \overline{T}_{k_1,k_2}(\ell_1,\ell_2)=R_{k_1}(\ell_1;a,bq^{2k_2-2j_2},u_1,j_1,q) R_{k_2}(\ell_2;aq^{2\ell_1-2j_1},b,u_2,j_2,q)\ , \label{eq:Tbar}\\
 &  \overline{S}_{k_1,k_2}(\ell_1,\ell_2)=R_{k_1}(\ell_1;a,b,u_1,j_1,q) R_{k_2}(\ell_2;aq^{2\ell_1-2j_1},bq^{2k_1-2j_1},u_2,j_2,q) \ .\label{eq:Sbar}
\end{align}
\end{thm}
\proof Recall the definitions \eqref{eq:defu}-\eqref{eq:defbw} with \eqref{eq:defpn}. The r.h.s. of \eqref{eq:overwu} evaluated on $(z_1,z_2)$ reads
\begin{align}
 \sum_{\ell_2=0}^{2j_2} 
 & R_{k_2}(\ell_2;a,bq^{2k_1-2j_1},u_2,j_2,q)
  p_{\ell_2}(z_2u_2q^{-j_2};a,j_2,q) \nonumber \\
 \times& \sum_{\ell_1=0}^{2j_1} 
  R_{k_1}(\ell_1;aq^{2\ell_2-2j_2},b,u_1,j_1,q)
  p_{\ell_1}(z_1u_1q^{-j_1};aq^{2\ell_2-2j_2},j_1,q)\,, 
\end{align}
which becomes, using \eqref{eq:overN1}  with \eqref{eq:un}, \eqref{eq:vn}:
\begin{align}
     \sum_{\ell_2=0}^{2j_2} 
 & R_{k_2}(\ell_2;a,bq^{2k_1-2j_1},u_2,j_2,q)
  p_{\ell_2}(z_2u_2q^{-j_2};a,j_2,q) p_{\ell_1}(z_1q^{j_1}/u_1;1/b,j_1,q^{-1})\,,
\end{align}
and, using again  \eqref{eq:overN1}, reduces to
\begin{align}
  p_{k_2}(z_2q^{j_2}/u_2;q^{2j_1-2k_1}/b,j_2,q^{-1}) p_{k_1}(z_1q^{j_1}/u_1;1/b,j_1,q^{-1})=\fw_{k_1,k_2}(z_1,z_2)\,.
\end{align}
The other equalities are proven similarly.
\endproof
These different overlap coefficients can also be visualized on Figure~\ref{fig:circ}.
 Note that the four overlap coefficients, given in the previous Theorem, are expressed as a product of two $q$-Racah univariate polynomials. Although both have six parameters $a,b,u_1,u_2,j_1,j_2$, they essentially differ by the structure of the entanglement of their arguments. We study in Subsection \ref{ssec:S} the functions $S_{k_1,k_2}$ and $\overline{S}_{k_1,k_2}$ and show their link with tridiagonal (TD) pairs. The functions $T_{k_1,k_2}$ and $\overline{T}_{k_1,k_2}$ are  studied in Subsection \ref{ssec:T}, and are connected to factorized Leonard (FL) pairs.

To conclude this section, let us mention that different bivariate analogs of the split basis \eqref{eq:sn} can be constructed. For instance, let us introduce the new sequence of bivariate polynomials
\begin{align}
    \fs_{n_1,n_2}(z_1,z_2)&=
    (z_1a q^{-j_1+2n_2-2j_2}u_1;q^2)_{n_1}
    (z_1b q^{j_1+2n_2-2j_2}/u_1;q^{-2})_{2j_1-n_1}\label{eq:sn12}\\
    &\qquad \times
    (z_2 a q^{-j_2}u_2;q^2)_{n_2}
    (z_2 b q^{j_2}/u_2;q^{-2})_{2j_2-n_2}\ .\nonumber
\end{align}
The following proposition is adapted from \cite[Ex.\,5.2]{BVZ16}. The proof being straightforward, we skip it.
\begin{prop}\label{prop:actAB}
The actions of ${\cal A}_{12}$ and ${\cal B}_{21}$ on the vectors $\fs_{n_1,n_2}$   take the following upper/lower diagonal structure, respectively:
\begin{align}  
\label{eq:A12snn}   \mathcal{A}_{12}\fs_{n_1,n_2}&=
    (q^{2n_1+2n_2-2j_1-2j_2}a+q^{2j_1+2j_2-2n_1-2n_2}a^{-1})\fs_{n_1,n_2}\\
    &\ \ + (q^{2j_1-n_1}-q^{-2j_1+n_1})(bq^{-2j_1-2j_2+2+2n_2+n_1}/u_1^2-q^{-2n_2-n_1+2j_2}/a)\fs_{n_1+1,n_2}\nonumber\\
     &\ \ +(q^{2j_2-n_2}-q^{-2j_2+n_2})(bq^{n_2-2j_1-2j_2+2}/u_2^2-q^{-n_2-2j_1}/a)
\fs_{n_1,n_2+1}\ ,\nonumber\\
\mathcal{B}_{21}\fs_{n_1,n_2}&=
(q^{2n_1+2n_2-2j_1-2j_2}b+q^{2j_1+2j_2-2n_1-2n_2}b^{-1})\fs_{n_1,n_2}\label{eq:B21snn}\\
&\ \ +(q^{n_1}-q^{-n_1})( q^{2j_1+2j_2-n_1-2n_2}/b-au_1^2q^{n_1+2n_2-2j_2-2} )\fs_{n_1-1,n_2}\nonumber\\
&\ \ +(q^{n_2}-q^{-n_2})(q^{2j_1+2j_2-n_2}/b-au_2^2 q^{n_2+2j_1-2})\fs_{n_1,n_2-1}\ .\nonumber
\end{align}
\end{prop}
For generic parameters $a,b,u_1,u_2$, $\{\fs_{n_1,n_2}\}_{n_1,n_2=0}^{2j_1,2j_2}$ form the so-called split basis of $V_{2j_1}(u_1)\otimes V_{2j_2}(u_2)$  (for details, see \cite[Lem.\,5.1]{BVZ16}).
\begin{rem}\label{rem:secsplit} A different sequence of bivariate polynomials $\bar\fs_{n_1,n_2}(z_1,z_2)$ 
can also be introduced on which $\mathcal{A}_{21}$ and $\mathcal{B}_{12}$ act as upper/lower diagonal matrices, respectively. Using ${\cal A}_{21}=(\sigma \circ {\cal A}_{12})|_{u_1\leftrightarrow u_2,j_1\leftrightarrow j_2,z_1\leftrightarrow z_2}$, ${\cal B}_{12}=(\sigma \circ {\cal B}_{21})|_{u_1\leftrightarrow u_2,j_1\leftrightarrow j_2,z_1\leftrightarrow z_2}$ and Proposition \ref{prop:actAB} we get
\beqa
\bar\fs_{n_1,n_2}(z_1,z_2)= \fs_{n_2,n_1}(z_2,z_1)|_{u_1\leftrightarrow u_2,j_1\leftrightarrow j_2}\ .
\eeqa
with \eqref{eq:sn12}.
Thus, $\{\bar\fs_{n_1,n_2}\}_{n_1=0,n_2=0}^{2j_1,2j_2}$  form a split basis of $V_{2j_1}(u_1)\otimes V_{2j_2}(u_2)$. 
\end{rem}

\section{Tridiagonal pairs, factorized Leonard pairs \& overlaps $S,T$} \label{sec:TDFLP}
In this Section, in a first part we show how the pairs of $q$-difference operators $({\cal A}_{12},{\cal B}_{21})$ and $({\cal A}_{21},{\cal B}_{12})$ provide explicit examples of TD pairs of $q$-Racah type, provided certain necessary and sufficient conditions on the scalars $a,b$ and the evaluation parameters $u_1,u_2$ are satisfied; See Proposition \ref{prop:TD}. The main characteristics of TD pairs under consideration -- namely the diameter, shape and character formula -- are computed, see Proposition \ref{prop:diamA12} and Corollary \ref{cor:charTD}. The overlap coefficients ${S}_{k_1,k_2}(\ell_1,\ell_2)$ relating the respective eigenbases for each TD pair are identified, see Proposition \ref{prop:overTD}. 
In a second part, we enlighten  connections  with bivariate polynomials of Tratnik type \cite{T91a,T91b,GR05,I08}, the rank $2$ Askey--Wilson algebra \cite{PW17,DD18,GW22,CF23} and the concept of FL pairs \cite{CZ23}. The main results are Theorem \ref{th:recu2} and Propositions \ref{prop:wAW2}.  It is also conjectured that, under certain conditions,  $(\mathcal{A}_2,\mathcal{A}_{12};\mathcal{B}_{12},\mathcal{B}_1)$ gives a basis for a FL pair of type $B'_2$,
see  Conjecture \ref{conj:FPBp2}.

\subsection{TD pairs of $q$-Racah type and the overlap coefficients $S_{k_1,k_2}(\ell_1,\ell_2)$ \label{ssec:S}} 
TD pairs can be viewed as generalizations of Leonard pairs, and were first introduced in \cite{ITT99}. 
\begin{defn} \cite{ITT99}
\label{def:tdrecall}
Let $V$ denote
a vector space over ${\mathbb K}$ with finite positive dimension.
By a {\it Tridiagonal pair} (or {\it TD pair}) on $V$
we mean an ordered pair ${\textsf A}, {\textsf A}^*$, where
${\textsf A}:V\rightarrow V$ and ${\textsf A}^*:V\rightarrow V$ are linear transformations 
satisfying (i)--(iv) below.
\begin{enumerate}[label=(\roman*)]
\item ${\textsf A}^*$ and ${\textsf A}^*$ are both diagonalizable on $V$.
\item There exists an ordering $V_0, V_1,\ldots, V_d$ of the  
eigenspaces of ${\textsf A}$ such that 
\begin{equation}
{\textsf A}^* V_p \subseteq V_{p-1} + V_p+ V_{p+1} \qquad \qquad (0 \leq p \leq d),
\label{eq:tdrecall1}\nonumber
\end{equation}
where $V_{-1} = 0$, $V_{d+1}= 0$.
\item There exists an ordering $V^*_0, V^*_1,\ldots, V^*_\delta$ of
the  
eigenspaces of ${\textsf A}^*$ such that 
\begin{equation}
{\textsf A} V^*_p \subseteq V^*_{p-1} + V^*_p+ V^*_{p+1} \qquad \qquad (0 \leq p \leq \delta),
\label{eq:tdrecall2}\nonumber
\end{equation}
where $V^*_{-1} = 0$, $V^*_{\delta+1}= 0$.
\item There is no subspace $W$ of $V$ such  that  both ${\textsf A}W\subseteq W$,
${\textsf A}^*W\subseteq W$, other than $W=0$ and $W=V$.
\end{enumerate}
We say the above TD pair is over ${\mathbb K}$.
\end{defn}
\begin{rem}\label{rem:diam}
In \cite{ITT99}, it is shown that $d=\delta$ which common value is called the diameter of the pair, and that $\dim(V_p)=\dim(V_p^*)$ for each $p$, which common value is denoted $\rho_p$. The sequence  $(\rho_0,\rho_1,...,\rho_d)$ is called the shape of the pair. The polynomial
\beqa
ch(\lambda) = \sum_{p=0}^{d} \rho_p \lambda^p  \label{eq:char}
\eeqa
is also introduced, and is called the character of the pair.
\end{rem} 

TD pairs are classified according to the structure of the sequences of respective eigenvalues of ${\textsf A},{\textsf A}^*$. For eigenvalue sequences of the form $\theta_p=e e^{2p}+f e^{-2p}$ and $\theta_p^*=e^* e^{2p}+f^* e^{-2p}$ for $e,f,e^*,f^*\in{\mathbb C}^*$,  respectively, the TD pair is said to be of {\it $q$-Racah type} 
\cite[Def.\,3.1]{IT08}. As shown in \cite[Thm.\,3.7]{Ter03},
for $q$ not a root of unity, the $q$-Onsager algebra $O_q$ gives an algebraic framework for TD pairs of $q$-Racah type.

From now on, we set ${\mathbb K}={\mathbb C}$. For the proof of the following theorem, we need 
the concept of ``$q$-strings'' associated with evaluation representations of $\Loop$ \cite{CP91}. Define the set called a $q$-string
\beqa
S(d,v) =\{vq^{2i-d+1}|0\leq i \leq d-1\}\quad \mbox{for $v\in \mathbb{C}^*$, $d\in {\mathbb N}$}\ .
\eeqa
Two $q$-strings $S(d,v),S(d',v')$ are said to be in {\it general position} if either $S(d,v)\cup S(d',v')$ does not produce\footnote{This happens if $v^{-1}v'=q^{\pm i}$ 
for some $i\in \{|\ell-\ell'|+2,|\ell-\ell'|+4,\cdots, \ell+\ell'\}$. In this case, $S(\ell,v)$ and $S(\ell',v')$
are said to be {\it adjacent}. 
} a $q$-string $S(d'',v'')$, or $S(d,v)\subseteq S(d',v')$ or  $S(d,v) \supseteq S(d',v')$.
\begin{prop}\label{prop:TD} The pairs $({\cal A}_{12},{\cal B}_{21})$ and $({\cal A}_{21},{\cal B}_{12})$ are TD pairs of $q$-Racah type if and only if 
\begin{enumerate}[label=(\roman*)]
\item For any $\varepsilon_1,\varepsilon_2\in\{\pm 1\}$, the $q$-strings $S(2j_1,(q/u_1^2)^{\varepsilon_1})$, $S(2j_2,(q/u_2^2)^{\varepsilon_2})$ are in general position.
\item For $i=1,2$, $ab,a/b \notin S(2j_i,q/u_i^2)\cup S(2j_i,(q/u_i^2)^{-1})$.
\item $\pm a,\pm b \notin \{q^{-2j_1-2j_2+1},\cdots,q^{2j_1+2j_2-1}\}$.
\end{enumerate}
\end{prop}
\begin{proof} Let us show that $({\cal A}_{12},{\cal B}_{21})$ is a TD pair of $q$-Racah type.
First, by Corollary \ref{cor:pairDG} and using \eqref{eq:Acalij}, the pair $({\cal A}_{12},{\cal B}_{21})$ satisfies the $q$-Dolan--Grady relations~\eqref{qDG1},~\eqref{qDG2} with $\rho \mapsto -(q^2-q^{-2})^2$; Second, recall from \eqref{eq:6bases} that each of $\{\fu_{n_1,n_2},\overline{\fw}_{n_1,n_2}\}$ form a basis of $V_{2j_1}(u_1)\otimes V_{2j_2}(u_2)$. Moreover, by \eqref{eq:specA12}, \eqref{eq:specB21}, $({\cal A}_{12},{\cal B}_{21})$ are diagonalizable for generic values of $a,b$, with respective eigenvalues of $q$-Racah type, as defined in \cite[Def.\,3.1]{IT08}; Third, recall that a set of conditions for which any tensor product of evaluation representations
is irreducible as a $O_q$-module has been given in \cite[Thm.\,1.17\, (i)]{IT10} (see also  \cite[Thm.\,5.17 (1)(i)-(iii)]{Ito2}). Although the classification theorem in \cite{IT10} uses an embedding $O_q \subset \Loop$ given by \cite[Eqs.\,(45) with eqs.(44)]{Ito2} that clearly differs from ours in Proposition \ref{prop:OqLoop}, the correspondence between our parameters $j_1,j_2,u_1,u_2,a,b$ and the parameters $\ell_1,\ell_2,a_1,a_2,s,t$ in \cite[Thm.\,1.17\, (i)]{IT10}  (or \cite[Thm.\,5.17]{Ito2}) can be established using the material given in \cite{Ito2}. Namely, on one hand, the embedding $O_q \subset \Loop$  considered in \cite{IT10,Ito2,Ito14} is given by $(W_0,W_1)\mapsto (z_t,z_t^*)$ with \cite[Eqs.\,(36),(37)]{Ito2}, and a split basis denoted $\{u_{n_2,n_1}|1\leq n_1\leq \ell_1,1\leq n_2\leq \ell_2\}$ is introduced \cite[Eq.\,(59)]{Ito2}. Thus, we compute the action of $z_t$ and $z_t^*$ on the vectors $u_{n_2,n_1}$ using \cite[Thm.\,3]{Ito2}, and compare the result with our expressions \eqref{eq:A12snn},\eqref{eq:B21snn}. This compararison shows that our split basis is related with the one in \cite{Ito2} as follows:
\beqa
u_{n_2,n_1} = c_{n_1,n_2}\fs_{n_1,n_2}\ ,
\eeqa
and the parameters are related as\footnote{Thus, we have the following correspondence of notations: $V(\ell_i,-q/u_i^2)|_{ \mbox{\cite{IT10,Ito2,Ito14}}}$ is $V_{2j_i}(u_i)$ in the present paper.\label{foot:VV}}
\begin{align}
\ell_1\to 2j_2\,, \ \ell_2\to 2j_1\,,\ a_1\to -q/u_2^2 \,,\ a_2\to -q/u_1^2\,,\ s \to \sqrt{ab}\,,\ t \to \sqrt{a/b}\ .\label{def:cor}
\end{align}
Meanwhile, we get the two constraints on the normalization coefficients 
\beqa
 \frac{c_{{n_1+1},n_2}}{c_{n_1,n_2}} &=& \frac{q^{n_1+1-2j_1}}{(q-q^{-1})}\frac{[2j_1-n_1]}{[n_1+1]}\sqrt{b/a}\ ,\quad \frac{c_{{n_1},n_2+1}}{c_{n_1,n_2}} = \frac{q^{n_2+1-2j_1-2j_2}}{(q-q^{-1})}\frac{[2j_2-n_2]}{[n_2+1]}\sqrt{b/a} \ ,
\eeqa
which solution is given by
\beqa
c_{n_1,n_2} = (q-q^{-1})^{-2(n_1+n_2)}\left(q^2b/a\right)^{(n_1+n_2)/2}q^{-2n_2j_1}\frac{(q^{-4j_1};q^2)_{n_1}
(q^{-4j_2};q^2)_{n_2}}{[n_1]![n_2]!}\ .
\eeqa
Having identified the relation  between our set of parameters and the ones in \cite{Ito2,Ito14}, we finally do the substitution \eqref{def:cor} in the conditions \cite[Thm.\,1.17\, (i)]{IT10},
which are found to be equivalent to the conditions (i)-(iii) of the claim. To summarize up to now, we have shown that $({\cal A}_{12},{\cal B}_{21})$ satisfy the $q$-Dolan-Grady relations, and are diagonalizable on the vector space:
\beqa
V:=V_{2j_1}(u_1)\otimes V_{2j_2}(u_2) \label{eq:vecspaceTD}
\eeqa
that is irreducible if and only if (i)--(iii) are satisfied.   
Alltogether and adapting \cite[Thm.\,3.7]{Ter03} to our choice of conventions, it implies that $({\cal A}_{12},{\cal B}_{21})$ is a TD pair of $q$-Racah type. 
The proof that $({\cal A}_{21},{\cal B}_{12})$ is a TD pair of $q$-Racah type is done along the same line.
\end{proof}
\begin{rem}\label{rem:blockact}The explicit action of ${\cal A}_{12}$ (resp.  ${\cal B}_{21}$) on the eigenbasis of ${\cal B}_{21}$   (resp. ${\cal A}_{12}$) can be computed using \eqref{eq:defcalA12} with \eqref{eq:defAi} and \eqref{eq:defu}, \eqref{eq:defbw}. For instance, the action of ${\cal B}_{21}$ on  $\fu_{n_1,n_2}$ (the eigenvectors of  ${\cal A}_{12}$)  is a linear combination of nine terms, namely $\fu_{n_1\pm 1,n_2}$, $\fu_{n_1,n_2\pm 1}$, $\fu_{n_1\pm 1,n_2\mp 2}$, $\fu_{n_1\pm 1,n_2\mp 1}$ and $\fu_{n_1,n_2}$. Recalling that the eigenvalues of ${\cal A}_{12}$ depend on $p:=n_1+n_2$ (see \eqref{eq:specA12}),  previous results show the block tridiagonal action of ${\cal B}_{21}$ accordingly to (ii) of Definition \ref{def:tdrecall}. Up to conventions, see \cite[Eq.\,(3.25)]{BVZ16} for details.
\end{rem}
Let us make a few comments regarding the conditions (i)-(iii) of Proposition \ref{prop:TD}.
For $\Loop$, it is known that every non-trivial finite-dimensional irreducible representation $V$ of type $(1,1)$ 
is isomorphic to a tensor product of evaluation representations. Moreover, there exists a one-to-one correspondence between isomorphisms 
classes of irreducible representations and so-called Drinfel'd polynomials, which encode the
irreducibility conditions \cite{CP91}.
For $O_q$, it is known that every non-trivial finite-dimensional irreducible representation $V$
is isomorphic to a tensor product of evaluation representations parametrized by two scalars\footnote{Denoted $s,t$ in \cite{IT10} and related with our scalars $a,b$ from \eqref{eq:paramab} using \eqref{def:cor}.}. The irreducibility conditions of $V$ for $O_q$
are given in \cite[Thm.\,1.17]{IT10}, that should be compared with \cite[Thm.\,4.8]{CP91} for $\Loop$.
In particular, to the condition \cite[(i1) in Thm.\,1.17]{IT10} is associated a Drinfel'd polynomial. Adapting its definition \cite[Thm.\,1.15(i)]{IT10} to our case \eqref{eq:vecspaceTD} using the correspondence \eqref{def:cor}, we get:
\beqa
P_V(\lambda) = \prod_{i=1}^2 \prod_{n=0}^{2j_i-1} (\lambda - u_i^2q^{-2n+2j_i-2} - u_i^{-2}q^{2n-2j_i+2})\ .
\eeqa
that is associated with the condition (i) of Proposition \ref{prop:TD}. Note that condition (ii) is encoded in $P_V(\lambda)\neq 0$ for $\lambda= ab+(ab)^{-1}$ and $\lambda= a/b+b/a$ (see \cite[Thm.\,1.11]{IT10}), and condition (iii) relies on the diagonalizability of the elements \eqref{def:A}, \eqref{def:B} on $V$ with \eqref{eq:paramab}.

The TD pairs associated with $({\cal A}_{12},{\cal B}_{21})$ and $({\cal A}_{21},{\cal B}_{12})$ have the following characteristics. Recall Remark \ref{rem:diam}.
\begin{prop}\label{prop:diamA12}
The TD pairs $({\cal A}_{12},{\cal B}_{21})$ and $({\cal A}_{21},{\cal B}_{12})$ with (i)-(iii) of Proposition \ref{prop:TD} have diameter $d=\delta=2j_1+2j_2$ and 
\beqa
\rho_p=\dim(V_p)=\dim(V_p^*)=\min(2j_1,p)+\min(2j_2,p) + 1 - p\ .\label{eq:dimVp}
\eeqa
\end{prop}
\begin{proof}
Let us consider the TD pair $({\cal A}_{12},{\cal B}_{21})$ with (i)-(iii) of Proposition \ref{prop:TD}. According to Definition \ref{def:tdrecall}, the vector space $V$ given by \eqref{eq:vecspaceTD} admits at least two different decompositions. 
Let $V_p$ (resp. $V^*_p$) denote the subspace of $V$ generated by the vectors $\fu_{n_1,n_2}$ (resp. $\overline{\fw}_{n_1,n_2}$)  with fixed values of $p$. From the eigenvalues \eqref{eq:specA12}, \eqref{eq:specB21} we identify $p=n_1+n_2$, where  $0 \leq n_1 \leq 2j_1$, $0 \leq n_2\leq 2j_2$. Recall the definition of the diameter of a TD pair in Remark \ref{rem:diam}. From (ii) and (iii) of Definition \ref{def:tdrecall}, we obtain $d=\delta=\max(p)=2j_1+2j_2$. The dimension of a given subspace $V_p,V_p^*$ is now computed. By definition, we have $n_1\in\big[0,2j_1\big]$. Using $n_2=p-n_1$, we also get $n_1\in\big[p-2j_2,p\big]$. It follows $n_1\in\big[\max(0,p-2j_2),\min(2j_1,p)\big]$. Counting the number of values of $n_1$ in this interval  and using $\min(2j_2,p) + \max(0,p-2j_2)=p$, after simplifications we obtain \eqref{eq:dimVp}. 
The same conclusion holds for the TD pair $({\cal A}_{21},{\cal B}_{12})$.
\end{proof}

Furthermore, recall the definition of the shape and character of a TD pair of $q$-Racah type in Remark \ref{rem:diam}. 
In \cite[Conj.\,13.5]{ITT99}, it is conjectured that given a TD pair, there exist integers $n$ and $\ell_1,\ell_2,\cdots,\ell_n$ such that the character of the pair is 
\beqa
ch(\lambda)&=& \prod_{i=1}^n \frac{(1-\lambda^{\ell_i+1})}{(1-\lambda)}\ .\label{eq:charform} 
\eeqa
Moreover, \eqref{eq:charform} is proven for a TD pair of $q$-Racah type on a $n$-tensor product of evaluation representations with labels $\ell_i$ 
\cite{IT10}:
In our case -- recall \eqref{def:cor} and footnote \ref{foot:VV}  -- we have \eqref{eq:vecspaceTD}. It follows:
\begin{cor} \label{cor:charTD}
The TD pairs $({\cal A}_{12},{\cal B}_{21})$ and $({\cal A}_{21},{\cal B}_{12})$ with (i)-(iii) of Proposition \ref{prop:TD} have a character formula given by \eqref{eq:charform} for $n=2$ and $\ell_1=2j_1,\ell_2=2j_2$.
\end{cor}
\begin{example} For $j_1=\h,j_2=1$, the corresponding TD pair has diameter $d=\delta=3$, shape given by $(1,2,2,1)$ and character:
\beqa
ch(\lambda)&=& 1+2\lambda + 2\lambda^2 +\lambda^3= \frac{(1-\lambda^2)}{(1-\lambda)}\frac{(1-\lambda^3)}{(1-\lambda)} \nonumber \ .
\eeqa
\end{example}

Finally, as a consequence of Proposition \ref{prop:TD} and Theorem \ref{thm:overlaps}, we obtain a bivariate generalization of \eqref{eq:overN1}, \eqref{eq:Rac1} upon the conditions given in Remark \ref{rem:LPcond}. 
\begin{prop}\label{prop:overTD} Assume the conditions (i)-(iii) hold. The overlap coefficients ${S}_{k_1,k_2}(\ell_1,\ell_2)$ (resp. $\overline{S}_{k_1,k_2}(\ell_1,\ell_2)$) between the respective eigenbases of the TD pair $({\cal A}_{12}$, ${\cal B}_{21})$ (resp. $({\cal A}_{21}$, ${\cal B}_{12})$) are given by \eqref{eq:S} (resp. \eqref{eq:Sbar}). 
\end{prop}

To conclude this section, let us make a few comments on the
notion of bispectrality for the overlap coefficients ${S}_{k_1,k_2}(\ell_1,\ell_2)$ (or similarly for $\overline{S}_{k_1,k_2}(\ell_1,\ell_2)$).
Recall that for the univariate case discussed in Subsection \ref{sec:qdiff}, the respective actions of ${\cal A},{\cal B}$ lead to a three-term recurrence relation \eqref{eq:diff1} and second order difference equation \eqref{eq:diffR2} satisfied by the $q$-Racah polynomials.
Similarly, a three-term ``block'' recurrence  relation and a second order ``block'' difference equation can be explicitly written 
for the $q$-Racah type coefficients  ${S}_{k_1,k_2}(\ell_1,\ell_2)$, using the explicit actions of ${\cal A}_{12}$ (resp. ${\cal B}_{21}$) on the eigenbasis of ${\cal B}_{21}$ (resp. ${\cal A}_{12}$) as mentioned in Remark \ref{rem:blockact}. However, contrary to the univariate case for which the recurrence and difference relations are sufficient to determine uniquely -- up to a normalization --$R_k(\ell)$, the two ``block'' relations are not sufficient to determine uniquely the functions ${S}_{k_1,k_2}(\ell_1,\ell_2)$. This is due to the existence of multiplicities in the eigenvalues of ${\cal A}_{12},{\cal B}_{21}$. To remove this ambiguity in the choice of eigenbases, it would be desirable to identify a second pair of $q$-difference operators, say $({\cal A}'_{12},{\cal B}'_{21})$
such that 
\beqa
[{\cal A}'_{12},{\cal A}_{12}]=0 \ ,\quad [{\cal B}'_{21},{\cal B}_{21}]=0\ ,
\eeqa
and that would remove the multiplicities.  Knowing their actions explicitly on the eigenbases of $\cal B_{21}$ and $\cal A_{12}$, 
we would get another pair of recurrence and difference relations for  ${S}_{k_1,k_2}(\ell_1,\ell_2)$. Although this problem goes beyond the scope of this paper, let us mention that
natural candidates for $({\cal A}'_{12},{\cal B}'_{21})$ are given by (images in $\End(V)$ of) the so-called {\it alternating generators} of $O_q$ \cite{BasBel,Ter21b}.  
   
Let us also mention \cite{CGT25} where the overlap coefficients for a TD pair of Racah type have been computed. A comparison between our coefficients ${S}_{k_1,k_2}(\ell_1,\ell_2)$ in the limit $q\rightarrow 1$ and the ones in \cite[Thm.\,5.1]{CGT25} shows that they do not match.   This indicates that the choice of eigenbases is different in the two approaches. It would be interesting but, also beyond the scope of this paper, to find the $q$-analogs of the overlap coefficients obtained in \cite{CGT25} that would be associated with the $q$-Racah type,  and find how they relate to the overlap coefficients ${S}_{k_1,k_2}(\ell_1,\ell_2)$.

\subsection{Bispectrality of $T_{k_1,k_2}(\ell_1,\ell_2)$, rank $2$ Askey-Wilson algebra and FL pairs}\label{ssec:T}
In this Subsection,
it is first shown that the overlap coefficient $T_{k_1,k_2}(\ell_1,\ell_2)$ at $u_1=u_2$ is proportional to a  bivariate polynomial on a discrete support, that enjoys bispectrality; See Proposition \ref{cor:polyprop} and  Theorem \ref{th:recu2}). Second, it is shown that the construction of Section \ref{sec:biloop} specialized to $u_1=u_2$ gives an homomorphism $\psi: AW(4)\rightarrow U_qsl_2 \otimes U_qsl_2$, see Proposition \ref{prop:wAW2}.  
Finally, we conjecture that our construction produces examples of $B'_2$-FL pairs, see Conjecture \ref{conj:FPBp2}.

\subsubsection{Bispectrality of $T_{k_1,k_2}(\ell_1,\ell_2)$}

 Under certain restrictions on the evaluation parameters $u_1,u_2$, the functions $T_{k_1,k_2}(\ell_1,\ell_2)$ are proportional to polynomials, as stated precisely in the following.
Using the explicit expression \eqref{eq:T} as well as the ones \eqref{eq:Rac1} of $R_k(\ell)$, $T_{k_1,k_2}(\ell_1,\ell_2)$ can be written as, for $u_1=u_2=:u$ and after some manipulations of the $q$-Pochhammer symbols:
\begin{align}
T_{k_1,k_2}(\ell_1,\ell_2)&\sim q^{-2k_1\ell_2}(abu^2q^{2\ell_2-2j_1-2j_2},abq^{2-2j_1-2j_2+2\ell_2}/u^2;q^2)_{k_1}\\
&\times{}_4\phi_3\left({{q^{-2k_1},\;q^{2k_1-4j_1}b^2, \;q^{-2\ell_1},\;q^{2\ell_1+4\ell_2-4j_1-4j_2}a^{2} }\atop
{q^{-4j_1},\;q^{2\ell_2-2j_1-2j_2} ab u^2, \; q^{2-2j_1-2j_2+2\ell_2}abu^{-2} }}\;\Bigg\vert \; q^2,q^2\right)\nonumber\\
   &\times {}_4\phi_3\left({{q^{-2k_2},\;q^{4k_1+2k_2-4j_1-4j_2}b^2, \;q^{-2\ell_2},\;q^{2\ell_2-4j_2}a^{2} }\atop
{q^{-4j_2},\;q^{2k_1-2j_1-2j_2} ab u^2, \; q^{2+2k_1-2j_1-2j_2}abu^{-2} }}\;\Bigg\vert \; q^2,q^2\right)\,,\nonumber
\end{align}
where $\sim$ stands for \textit{equal to up to functions $f(k_1,k_2)$ and $g(\ell_1,\ell_2)$}. Let us perform a Sear's transformation \cite[(III.15)]{GR04} on the first $_{4}\phi_3$ appearing in previous relation, to get
\begin{align}
\label{eq:Tpoly}T_{k_1,k_2}(\ell_1,\ell_2)&\sim (abu^2q^{2\ell_2-2j_1-2j_2},bu^2q^{-2j_1+2j_2-2\ell_2}/a;q^2)_{k_1}\\
\times&{}_4\phi_3\left({{q^{-2k_1},\;q^{2k_1-4j_1}b^2, \;abu^2 q^{2\ell_1+2\ell_2-2j_1-2j_2},\;bu^2q^{2j_1+2j_2-2\ell_1-2\ell_2}/a }\atop
{b^2q^{2},\;abu^2q^{2\ell_2-2j_1-2j_2}, \; bu^2q^{-2j_1+2j_2-2\ell_2}/a}}\;\Bigg\vert \; q^2,q^2\right)\nonumber\\
   \times& {}_4\phi_3\left({{q^{-2k_2},\;q^{4k_1+2k_2-4j_1-4j_2}b^2, \;q^{-2\ell_2},\;q^{2\ell_2-4j_2}a^{2} }\atop
{q^{-4j_2},\;q^{2k_1-2j_1-2j_2} ab u^2, \; q^{2+2k_1-2j_1-2j_2}abu^{-2} }}\;\Bigg\vert \; q^2,q^2\right)\,.\nonumber
\end{align}
This result leads to the following proposition.
\begin{prop}\label{cor:polyprop}  Assume $u_1=u_2:=u$. Then,
${T}_{k_1,k_2}(\ell_1,\ell_2)$ is equal, up to functions $f(k_1,k_2)$ and $g(\ell_1,\ell_2)$, to a bivariate polynomial in the variables
\beqa
\nu=q^{2\ell_2-2j_2}a+q^{2j_2-2\ell_2}/a\qquad \mbox{and}\qquad \mu= q^{2\ell_1+2\ell_2-2j_1-2j_2}a+q^{2j_1+2j_2-2\ell_1-2\ell_2}/a\,.\label{eq:specdata}
\eeqa
This polynomial is given by the r.h.s. of \eqref{eq:Tpoly}, has five parameters $a,b,u,j_1,j_2$, and is of total degree $k_1+k_2$ and of degree $k_1$ with respect to $\mu$.
\end{prop}
\proof The second $_{4}\phi_3$ in the r.h.s. of \eqref{eq:Tpoly} is a polynomial of degree $k_2$ with respect to $\nu$. The product of the $q$-Pochhammer symbols and the first $_{4}\phi_3$ in the r.h.s. of \eqref{eq:Tpoly} is a polynomial of total degree $k_1$ with respect to $\lambda$ and $\mu$, which concludes the proof.  \endproof

As demonstrated in Theorem~\ref{thm:overlaps}, the functions $T_{k_1,k_2}(\ell_1,\ell_2)$ are the overlap coefficients between $\fu_{n_1,n_2}$ and $\fw_{n_1,n_2}$, the eigenbases  of $({\cal A}_2,{\cal A}_{12})$ and $({\cal B}_1,{\cal B}_{12})$, respectively. 
As recalled in Subsection~\ref{sec:qdiff} for the univariate case, properties of the overlap coefficients can be obtained using the different actions of the  operators $\cal A$, $\cal B$. The same trick is used here to study properties of   $T_{k_1,k_2}$. 
We compute in the following propositions different actions of the $q$-difference operators  ${\cal A}_2,\, {\cal A}_{12},\,{\cal B}_1$ and ${\cal B}_{12}$ in different bases.
\begin{prop}
   The following actions hold
    \begin{align}
\label{eq:A2w}      {\cal A}_2  \fw_{n_1,n_2}&= {\cal A}^{(2)}_{n_2-1,n_2} \fw_{n_1,n_2-1}+{\cal A}^{(2)}_{n_2,n_2}\fw_{n_1,n_2}+{\cal A}^{(2)}_{n_2+1,n_2}\fw_{n_1,n_2+1}\,,\\
\label{eq:B1u}      {\cal B}_1    \fu_{n_1,n_2} &= {\cal B}^{(1)}_{n_1-1,n_1} \fu_{n_1-1,n_2}+{\cal B}^{(1)}_{n_1,n_1} \fu_{n_1,n_2}+{\cal B}^{(1)}_{n_1+1,n_1} \fu_{n_1+1,n_2}\,,
    \end{align}
    where ${\cal A}^{(2)}_{i,j}={\cal A}_{i,j}\big|_{b\to bq^{2n_1-2j_1},u\to u_2}$ defined by \eqref{eq:Amn}-\eqref{eq:A0n} and  ${\cal B}^{(1)}_{i,j}={\cal B}_{i,j}\big|_{a\to aq^{2n_2-2j_2},u\to u_1}$ defined by \eqref{eq:Bmn}-\eqref{eq:B0n}.
\end{prop}
\proof  The two relations \eqref{eq:A2w}, \eqref{eq:B1u} are direct consequences of \eqref{eq:tri1} and \eqref{eq:tri2}. \endproof

 Actually, there do not exist simple actions for ${\cal A}_{12}$ on $\fw_{n_1,n_2}$ and for ${\cal B}_{12}$ on $\fu_{n_1,n_2}$ in general. Using formal mathematical software, we do not find actions involving a linear combinations of all the terms $\fw_{n_1+r,n_2+s}$ and $\fu_{n_1+r,n_2+s}$ with $|r|,|s|,|r\pm s|\leq 4$. However, if both evaluation parameters are equal ($u_1=u_2$), these actions simplify as given in the following proposition.
\begin{prop}\label{pr:A12B12}
 For $u_1=u_2:=u$, the following actions hold
    \begin{align}
 \label{eq:A12w}       {\cal A}_{12}\fw_{n_1,n_2}=&\sum_{(r,s)\in B_2'}{\cal A}^{(r,s)}_{n_1,n_2}\fw_{n_1+r,n_2+s}\,,\\
 \label{eq:B12u}       {\cal B}_{12}   \fu_{n_1,n_2}=&\sum_{(s,r)\in B_2'}{\cal B}^{(r,s)}_{n_1,n_2}\fu_{n_1+r,n_2+s}\,,
    \end{align}
   for some coefficients $\cal A_{n_1,n_2}^{(r,s)},\cal B_{n_1,n_2}^{(r,s)}$  and where $B'_2$ is the following subset of $\mathbb{Z}^2$:
\begin{align} \label{eq:AB2}
 &B'_2=\{(-1,2),(-1,1),(0,1), (-1,0),(0,0), (1,0),  (0,-1), (1,-1), (1,-2)\}\,.
 \end{align}
\end{prop}
\proof The proof follows the one of Proposition~\ref{pr:eigA2B21}: act with the $q$-difference operator on the $q$-Pochhammer symbols, use their properties to write them in terms of the symbols we started with. One finds that the r.h.s.\,of \eqref{eq:A12w}, \eqref{eq:B12u} takes the form displayed.  The explicit expressions for $\cal A_{n_1,n_2}^{(r,s)},\cal B_{n_1,n_2}^{(r,s)}$ are not reported to enlighten the presentation.  \endproof
We are now in position to provide a set of relations satisfied by $T_{k_1,k_2}(\ell_1,\ell_2)$.
\begin{thm}\label{th:recu2}
    Assume $u_1=u_2:=u$. The overlap coefficients $T_{k_1,k_2}(\ell_1,\ell_2)$ satisfy the following two recurrence relations:
\begin{align}
 &   (q^{2\ell_2-2j_2}a+q^{2j_2-2\ell_2}/a)T_{k_1,k_2}(\ell_1,\ell_2)=\sum_{r=0,\pm 1}
    {\cal A}^{(2)}_{k_2+r,k_2} T_{k_1,k_2+r}(\ell_1,\ell_2)\,,\label{eq:rec1T}\\
  &  (q^{2\ell_1+2\ell_2-2j_1-2j_2}a+q^{2j_1+2j_2-2\ell_1-2\ell_2}/a)T_{k_1,k_2}(\ell_1,\ell_2)=\sum_{(r,s)\in B_2'}{\cal A}^{(r,s)}_{k_1,k_2}T_{k_1+r,k_2+s}(\ell_1,\ell_2)\,,\label{eq:rec2T}
\end{align}
and two difference equations:
\begin{align}
&  (q^{2k_1-2j_1}b+q^{2j_1-2k_1}/b) T_{k_1,k_2}(\ell_1,\ell_2)=\sum_{r=0,\pm 1}{\cal B}^{(1)}_{\ell_1,\ell_1+r}T_{k_1,k_2}(\ell_1+r,\ell_2)\,,\\
 \label{eq:diff2}  &(q^{2k_1+2k_2-2j_1-2j_2}b+q^{2j_1+2j_2-2k_1-2k_2}/b)T_{k_1,k_2}(\ell_1,\ell_2)=\sum_{(s,r)\in B_2'}{\cal B}^{(r,s)}_{\ell_1-r,\ell_2-s}T_{k_1,k_2}(\ell_1-r,\ell_2-s) \,.
\end{align}
\end{thm}
\proof Relations \eqref{eq:A2w}-\eqref{eq:B12u} allow us to obtain two recurrence relations and two difference equations for the coefficients $T_{k_1,k_2}(\ell_1,\ell_2)$. Indeed, acting with the operators $\cal A_2$ (with \eqref{eq:A2u}, \eqref{eq:A2w}), $\cal A_{12}$ (with \eqref{eq:specA12}, \eqref{eq:A12w}), $\cal B_1$ (with \eqref{eq:B1w}, \eqref{eq:B1u}) and $\cal B_{12}$ (with \eqref{eq:B12w} and \eqref{eq:B12u}) on \eqref{eq:overwu}, we prove the relations of the theorem.
\endproof
 Accordingly to \eqref{eq:rec1T}-\eqref{eq:diff2}, the coefficients $T_{k_1,k_2}(\ell_1,\ell_2)$ at $u_1=u_2$ are said to be bispectral. 
Let us mention that
the overlap coefficients  $\overline{T}_{k_1,k_2}(\ell_1,\ell_2)$ defined in \eqref{eq:Tbar} at $u_1=u_2$ satisfy similar relations than the ones for ${T}_{k_1,k_2}(\ell_1,\ell_2)$. This is shown using the actions of $\cal A_1$, $\cal B_2$, $\cal A_{21}$ and $\cal B_{21}$. These relations are omitted here, as they do not offer new insights into our findings.\vspace{1mm}

Finally, let us discuss potential connections between the polynomials introduced in Proposition \ref{cor:polyprop}, and multivariate polynomials of Tratnik type defined on a discrete support. On one hand,  the polynomials in Proposition \ref{cor:polyprop}
have five parameters and spectral data \eqref{eq:specdata}. Also, it can be easily shown that they enjoy bispectrality using  \eqref{eq:rec1T}--\eqref{eq:diff2}. On the other hand, in \cite{T91a} the multivariate polynomials of Wilson type have been studied, then in \cite{T91b} the discrete families have been introduced. 
The $q$-analogs of the latter are the so-called Gasper--Rahman polynomials \cite{GR05} whose bispectrality has been proven in \cite{I08}.
For the bivariate case, they have five parameters and their spectral data matches with \eqref{eq:specdata}. This strongly suggests that the polynomials introduced in Proposition \ref{cor:polyprop} and the bivariate Gasper--Rahman polynomials are closely related, although their precise relationship is not clear at the moment.

\subsubsection{Bispectral algebra and rank 2 Askey--Wilson algebra}
In Subsections~\ref{sec:AW3} and \ref{sec:qdiff}, we recalled that the bispectral algebra associated to the $q$-Racah polynomials is the Askey--Wilson algebra $AW(3)$. Similarly, the bispectral algebra associated to the overlap coefficients $T_{k_1,k_2}(\ell_1,\ell_2)$  given by \eqref{eq:T} is given by the rank 2 Askey--Wilson algebra\footnote{This higher rank version of Askey--Wilson algebra has been introduced and studied in \cite{PW17,DD18,GW22,CF23}.} $AW(4)$, as we now show.

\begin{defn}
The Askey--Wilson algebra of rank 2 denoted $AW(4)$ is the unital associative ${\mathbb C}$-algebra generated by the elements  $G_1,\, G_2,\, G_3,\, G_4,\,G_{12},\, G_{23},\, G_{34}\,, G_{123}\,,G_{234},\,G_{1234}$. One defines the following elements, for $(I,J,K)\in\{(1,2,3),(2,3,4),(12,3,4),(1,23,4),(1,2,34)\}$:
\begin{align}
   & G_{IK}:=-\frac{1}{q^2-q^{-2}}\lbrack G_{IJ},G_{JK} \rbrack_q +\frac{1}{q+q^{-1}}(G_IG_K+G_JG_{IJK})\,,\\
    & G_{KI}:=-\frac{1}{q^2-q^{-2}}\lbrack G_{JK},G_{IJ} \rbrack_q +\frac{1}{q+q^{-1}}(G_IG_K+G_JG_{IJK})\,.
\end{align}
The defining relations are:
\begin{itemize}
    \item For $I,J\in\{1,2,3,4,12,23,34,123,234,1234\}$,
    \begin{align}
        [G_I,G_J]=0\quad\text{if }\quad I\subset J\quad\text{or}\qquad  I\cap J=\emptyset\,,
    \end{align}
    \item For $(I,J,K)\in\{(1,2,3),(2,3,4),(12,3,4),(1,23,4),(1,2,34)\}$,
    \begin{align}
        &G_{IJ}+	\frac{1}{q^2-q^{-2}} \lbrack G_{JK}, G_{IK}\rbrack_q   =\frac{1}{q+q^{-1}}(G_{I}G_{J}+G_{K}G_{IJK})\,,\\
       & G_{JK}+	\frac{1}{q^2-q^{-2}} \lbrack G_{IJ}, G_{KI}\rbrack_q   =\frac{1}{q+q^{-1}}(G_{J}G_{K}+G_{I}G_{IJK})\,, 
    \end{align}
    \item\begin{align}
&G_{14}+	\frac{1}{q^2-q^{-2}} \lbrack G_{13}, G_{34}\rbrack_q   =\frac{1}{q+q^{-1}}(G_{1}G_{4}+G_{3}G_{134}) \ , \label{AW43}\\
&G_{41}+	\frac{1}{q^2-q^{-2}} \lbrack G_{42}, G_{12}\rbrack_q   =\frac{1}{q+q^{-1}}(G_{1}G_{4}+G_{2}G_{412}) \ , \label{AW44}
\end{align} 
\end{itemize}
\end{defn}

Let us recall that 
 $c_1$ is the element of $\Loop$ defined in~\eqref{def:Cas}. Using the map $\Delta_{u_1,u_2}$ defined by \eqref{eq:Duu}, we define:
\beqa
{C}_{12}:= {\Delta}_{u,u}(c_1)\,.
\eeqa

\begin{prop} \label{prop:wAW2} 
For $u_1=u_2:=u$, there exists an algebra homomorphism $\psi: AW(4) \rightarrow \Uq^{\otimes 2}$ such that:
	\beqa
\psi:	&&G_{12}\mapsto C_{12}\,,\quad G_{123} \mapsto  A_{12}\ ,\quad G_{23} \mapsto  A_2 \ ,\quad G_{234} \mapsto  B_1\ ,\quad G_{34} \mapsto  B_{12}\ , \\
	&& G_{1} \mapsto C\otimes 1\ , \quad G_2 \mapsto  1\otimes C\ , \quad G_3 \mapsto  \alpha\ ,\quad G_4 \mapsto q^{-1}u^2+qu^{-2}\ ,\quad G_{1234} \mapsto \beta\ , \nonumber
	\eeqa
	where $C$, $A_2$, $B_1$, $A_{12}$ and $B_{12}$ are defined by \eqref{def:Cas}, \eqref{eq:defA2}, \eqref{eq:defAs2}, \eqref{eq:defA12} and \eqref{eq:defB12}.
\end{prop}
\proof The proof is based on a verification of all the defining relations of $AW(4)$ expressing  $C_{12}$, $A_{12},...$ in terms of the generators of $\Uq$ and using the commutation relations of $\Uq$.  \endproof

The composition of homomorphisms $(\pi_{j_1}\otimes \pi_{j_2}) \circ \psi$ provides a realization of $AW(4)$ in terms of $q$-difference operators parametrized by $u$, and acting on $\mathbb{C}_{2j_1}[z_1]\otimes \mathbb{C}_{2j_2}[z_2]$. This proves that the bispectral algebra associated to the overlap coefficients $T_{k_1,k_2}(\ell_1,\ell_2)$ is $AW(4)$.

\subsubsection{Link with the bialgebra structure of $\Uq$}
The bialgebra structure of $\Uq$ is equipped with the coproduct $\delta \colon \Uq \to \Uq \otimes \Uq$ given by:
	\begin{align} \label{def:coprod}
		&\delta(E)= E \otimes K^{-1}  + K \otimes E , \qquad \delta(F) =F \otimes K^{-1} + K \otimes F, \qquad \delta(K^{\pm 1}) = K^{ \pm 1} \otimes K^{\pm 1} \,. 
	\end{align}
	We also introduce the opposite coproduct of $\Uq$:
	\beqa
	\bar{\delta}  = \sigma \circ \delta \ .
	\eeqa
Then, we observe that there exist relations between the actions of the coproduct of $\Loop$ and the ones of $\Uq$, through the evaluation homomorphism \eqref{map:eval}. 
\begin{lem}\label{lem:Deldel} The following relations hold in $\Uq\otimes \Uq$:
\begin{align}
    A_{12}\big|_{u_1=u_2=u}=\Delta_{u,u}(\widehat{A})=\delta(A)\,,\\
\label{eq:Dbd}      B_{21}\big|_{u_1=u_2=u}=\Delta_{u,u}(\widehat{B})=\bar{\delta}(B)\,,\\
        A_{21}\big|_{u_1=u_2=u}=\overline{\Delta}_{u,u}(\widehat{A})=\bar{\delta}(A)\,,\\
  \label{eq:bDd}          B_{12}\big|_{u_1=u_2=u}=\overline{\Delta}_{u,u}(\widehat{B})=\delta(B)\,,
    \end{align}
    where $\widehat{A},\;\widehat{B}$ are given by \eqref{def:A}-\eqref{def:B} and $A,B$  by \eqref{eq:w0u}-\eqref{eq:w1u}.
\end{lem}
\proof By direct computation of both sides of each equality. \endproof
Let us emphasize that, in \eqref{eq:Dbd}, it is the action of the coproduct $\Delta$ which provides the same result than the opposite coproduct $\bar{\delta}$ (and vice-versa in \eqref{eq:bDd}).
The previous lemma explains how the results obtained in \cite{G25} using only the bialgebra structure of $\Uq$ are similar to the ones obtained here starting from the bialgebra structure of $\Loop$. Indeed, in \cite{G25} the operators $\delta(A)$ and $\delta(B)$ are studied, as well as the overlap coefficients between their respective eigenbases.

This lemma allows us to establish further relations at $u_1=u_2:=u$.
\begin{prop}\label{prop:A12B12bis} The pairs $(\cal A_{12},\cal B_{12})$ and $(\cal A_{21},\cal B_{21})$ at $u_1=u_2:=u$ satisfy the $q$-Dolan--Grady relations~\eqref{qDG1},~\eqref{qDG2} with $\rho \mapsto -(q^2-q^{-2})^2$.
\end{prop}
\begin{proof} Applying $\textsf{ev}_u \circ \phi$ to \eqref{qDG1}, \eqref{qDG2}, one gets that the pair $(A,B)$ given by \eqref{eq:w0u}, \eqref{eq:w1u} satisfies the $q$-Dolan--Grady relations. Then, we consider the image of $(A,B)$ via $\delta$ (resp. $\bar\delta$). The pair $(\delta(A),\delta(B))$ (resp. $(\bar\delta(A),\bar\delta(B)$) satisfies the $q$-Dolan-Grady relations as well. Using Lemma \ref{lem:Deldel} and eq.\,\eqref{eq:Acalij}, the claim follows.
\end{proof}

\subsubsection{Factorized Leonard pairs} \label{sec:FLP}
In previous sections, 
it is  recalled that $q$-Racah polynomials are associated to the notion of Leonard pair, and that the coefficients $S_{k_1,k_2}(\ell_1,\ell_2)$ are related to TD pairs. In this section, it is conjectured that the coefficients $T_{k_1,k_2}(\ell_1,\ell_2)$ are also associated to TD pairs, but exhibit a refined structure associated with the notion of FL pairs introduced in \cite{CZ23}.

Assume the vector space $V_{2j_1}(u)\otimes V_{2j_2}(u)$ is an irreducible $({\cal A}_{12},{\cal B}_{12})$--module. By \cite[Thm.\,3.7]{Ter03}, according to
\eqref{eq:specA12}, \eqref{eq:B12w} and Proposition \ref{prop:A12B12bis}, it follows that  $({\cal A}_{12},{\cal B}_{12})$ is a particular TD pair of $q$-Racah type. A refined structure of this TD pair can be exhibited using the concept of  FL pairs, as we now show. In \cite{CZ23} it has been defined a FL pair of type $A_2$ in a triangular domain, called a factorized $A_2$-Leonard pair \cite[Def 3.3]{CZ23}.  Here, we extend this definition in a straightforward manner, by introducing a factorized $B'_2$-Leonard pair that is associated with a rectangular domain  
$\cD$ given by:
\begin{equation}\label{eq:dom}
 \cD=\{ (x,y)\ | \ x,y\in \mathbb{Z}_{\geq 0}, \ 0\leq x\leq j_1, 0\leq y\leq j_2 \}\,.
\end{equation}
Two couples $(x,y),(x',y')\in \cD$ are called $B'_2$-adjacent, $(x,y)\adj(x',y')$, if
\begin{equation}
  (x-x',y-y') \in B'_2\,,
\end{equation}
where $B'_2$ is the set defined by \eqref{eq:AB2}.
\begin{defi}\label{def:IT}
Let $V$ be a $\mathbb{C}$-vector space.
The pair $(H,H^\star)$ is a $B'_2$-Leonard pair on the domain $\cD$ if the following statements are satisfied:
\begin{itemize}
 \item[(i)] $H$ is a two-dimensional  subspace of $\End(V)$ whose elements are diagonalizable and mutually commute;
 \item[(ii)] $H^\star$ is a two-dimensional  subspace of $\End(V)$ whose elements are diagonalizable and mutually commute;
 \item[(iii)] there exists a bijection $\lambda\mapsto V_\lambda$  from $\cD$ to the common eigenspaces of $H$ such that, for all $\lambda\in \cD$,
 \begin{equation}\label{eq:HsV}
  H^\star V_\lambda \subseteq \sum_{ \genfrac{}{}{0pt}{}{\chi \in \cD}{\bar \chi \adj   \bar \lambda}} V_\chi\,,
 \end{equation}
 where $\bar \chi$ (resp. $\bar \lambda$) denotes the pair obtained by reversing the order of the entries in $ \chi$ (resp. $ \lambda$);
 \item[(iv)]  there exists a bijection $\lambda\mapsto V^\star_\lambda$  from $\cD$ to the common eigenspaces $V^\star_\lambda$ of $H^\star$ such that, for all $\lambda\in \cD$,
 \begin{equation}\label{eq:HVs}
  H V^\star_\lambda \subseteq \sum_{ \genfrac{}{}{0pt}{}{\chi \in \cD}{\chi \adj   \lambda}} V^\star_\chi\,;
 \end{equation}
 \item[(v)] there does not exist a subspace $W$ of $V$ such that $HW\subseteq W$, $H^\star W\subseteq W$, $W\neq 0$, $W\neq V$;
 \item[(vi)] $\dim(V_\lambda)=\dim(V^\star_\lambda)=1$, for all  $\lambda\in \cD$;
 \item[(vii)] there exists a non-degenerate symmetric bilinear form  $\langle,\rangle$ on $V$ such that, for $\lambda,\chi\in \cD$ and $\lambda\neq \chi$,
 \begin{equation}
\langle V_\lambda, V_\chi\rangle=0\ ,\qquad  \langle V^\star_\lambda, V^\star_\chi\rangle=0\,.
 \end{equation}
\end{itemize}
\end{defi}

\begin{defi}\label{def:flp}
Let $(H,H^\star)$ be a $B'_2$-Leonard pair on the domain $\cD$ and $X,Y\in \End(V)$ (resp. $X^\star,Y^\star$) be a basis of $H$ (resp. $H^\star$).
The pair $(H,H^\star)$ is a \textit{factorized $B'_2$-Leonard pair} if the following statements are satisfied:
\begin{itemize}
 \item[(viii)] the eigenvalues of $X$ and $Y^\star$ satisfy, for any $(x,y),(i,j)\in \cD$,
 \begin{align}\label{eq:eigenXY}
 & (X-t_x 1) V_{(x,y)} =0 \,,  \qquad 
 (Y^\star-\theta^\star_j 1) V^\star_{(i,j)} =0 \,,
  \end{align}
 $t_x\neq t_{x'} \Leftrightarrow x\neq x'$ 
  and $\theta^\star_j\neq \theta^\star_{j'} \Leftrightarrow j\neq j'$;
 \item[(ix)] in addition to the commutation relations $[X,Y]=0$ and $[X^\star,Y^\star]=0$, one has
 \begin{align}\label{eq:com1}
  [X,Y^\star]=0\,;
 \end{align}
 \item[(x)] each common eigenspace of $X$ and $Y^\star$ has dimension one;
   \item[(xi)]   
  the following relations hold, for $(i,j),(i+1,j)\in \cD$ and $(x,y),(x,y+1)\in \cD$,
 \begin{subequations}
      \begin{align}
 & \langle\, V^\star_{(i+1,j)}\, ,\, X V^\star_{(i,j)}\,\rangle\neq 0\,,&& \langle\, V^\star_{(i,j)}\, ,\, X V^\star_{(i+1,j)}\,\rangle\neq 0\,, \label{eq:irre1}\\
      &\langle V_{(x,y+1)}\,,\, Y^\star V_{(x,y)}\rangle \neq 0\,,&& \langle V_{(x,y)}\,,\, Y^\star V_{(x,y+1)}\rangle \neq 0\,.\label{eq:irre2}
  \end{align} 
  \end{subequations}
 \end{itemize}
 The ordered 4-tuple $(X,Y;X^\star,Y^\star)$ is called a basis of the factorized $B'_2$-Leonard pair.
 \end{defi}
\begin{rem}
Note that the definition of $A_2$- and $B'_2$-Leonard pairs is closely related to the family of linear maps introduced in \cite[Problem\,7.1]{TI10}.
\end{rem}

\begin{conj}\label{conj:FPBp2} Assume the vector space $V_{2j_1}(u)\otimes V_{2j_2}(u)$ is irreducible with respect to the action of the 4-tuple $(\mathcal{A}_2,\mathcal{A}_{12};\mathcal{B}_{12},\mathcal{B}_1)$. Then, $(\mathcal{A}_2,\mathcal{A}_{12};\mathcal{B}_{12},\mathcal{B}_1)$ is a basis of a factorized $B'_2$-Leonard pair.
\end{conj}
If this conjecture is true, then the coefficients $T_{k_1,k_2}(\ell_1,\ell_2)$ at $u_1=u_2$ find a natural interpretation as overlap coefficients relating different eigenbases associated with a factorized $B'_2$-Leonard pair. 

 \section{Concluding remarks} 
 A few comments and perspectives are now highlighted.
Firstly, let us point out that Proposition \ref{prop:overTD} gives an explicit answer to the question rised by P. Terwilliger in \cite[Above Problem 2.3]{Ter03} for TD pairs of type I (also called $q$-Racah type) which main characteristics are given in Proposition \ref{prop:diamA12} and Corollary \ref{cor:charTD}. We are confident that the analysis for the bivariate case $N=2$ can be extended to the multivariate case $N>2$, combining the results in \cite{IT10} and \cite{BVZ16}. Under certain irreducibility conditions that remain to be identified, we thus conjecture that the expected multivariate $q$-Racah type functions have  $2N+2$ parameters ($a,b$ and $\{u_i,j_i\}_{i=1}^{N}$) and are, up to conventions, already given in \cite{BVZ16}. For the case $q=1$, we mention the work \cite{CGT25} where overlap coefficients associated with TD pairs of type II (also called Racah type) are obtained. How our overlap coefficients in Theorem \ref{thm:overlaps} (eq.\,\eqref{eq:overwu}) at $q=1$ relate to those in \cite[Thm.\,5.1]{CGT25} is an open problem. For TD pairs of type III (also called Bannai-Ito type), the corresponding analog of the $q$-Onsager and Racah algebra  \cite{GZ88} is  known \cite{BGTVZ14}, but the question of overlap coefficients remains open.

Secondly, in \cite{CZ23} recall that
the concept of FL pairs is introduced  and studied for the root system $A_2$. 
In the conclusion of that paper, the existence of FL pairs associated with different root systems is conjectured. In the present paper, provided certain irreducibility conditions are satisfied,   TD pairs associated with $({\cal A}_{12}$, ${\cal B}_{21})$ (or $({\cal A}_{21}$, ${\cal B}_{12})$) lead to FL pairs for the root system $B'_2$.  As pointed out below Proposition \ref{prop:overTD}, it would be also desirable to exhibit a similar refined structure for the TD pairs $({\cal A}_{12}$, ${\cal B}_{21})$ (or $({\cal A}_{21}$, ${\cal B}_{12}))$. Indeed, given a TD pair, multiplicities in the eigenvalues suggest the existence of a refined algebraic structure that would allow to further distinguish TD pairs with the same eigenvalues sequences, diameter, shape and character formula.  

Thirdly, Tratnik polynomials admit generalizations known as Griffiths polynomials. Originally introduced in \cite{G71}, these were initially of Krawtchouk type, since they can be written as sums of products of three univariate Krawtchouk polynomials \cite{GVZ13}. Recent work has extended Griffiths polynomials to the Racah \cite{CFR23,CFGRVZ24} and $q$-Racah \cite{CFGR25} families. Bivariate Tratnik polynomials emerge as limiting cases of the Griffiths polynomials. The general framework introduced in this paper, coupled with connections to quantum groups and quadratic algebras, provides a clear pathway toward a deeper understanding of Griffiths polynomials.

\vspace{5mm}

\noindent
{\bf Acknowledgments:} We thank Paul Terwilliger for communications related with \cite{IT10}.
P.B. and N.C.  are supported by C.N.R.S.

\end{document}